\documentclass[11pt,reqno]{amsart}

\usepackage{mathrsfs}
\usepackage{amsmath}
\usepackage{amssymb}
\usepackage{amsfonts}
\usepackage{amsthm}
\usepackage{a4wide}
\usepackage{graphicx}
\usepackage{color}
\usepackage[sans]{dsfont}
\usepackage{pgfkeys}
\usepackage{longtable}
\usepackage{pdflscape}
\usepackage{float}
\usepackage{cancel}
\usepackage{comment}

\usepackage{dsfont}
\usepackage{setspace}
\usepackage{array}
\usepackage{hyperref}
\hypersetup{colorlinks=true,allcolors=[rgb]{0,0,0.6}}
\usepackage[english]{babel}
\usepackage[latin1]{inputenc}
\usepackage[T1]{fontenc}
\usepackage{enumitem}
\usepackage[square,numbers,sort&compress]{natbib}
\usepackage[shadow,textwidth=22mm,textsize=tiny]{todonotes}

\newtheorem{theorem}{Theorem}[section]
\newtheorem{lemma}[theorem]{Lemma}
\newtheorem{proposition}[theorem]{Proposition}

\newtheorem{definition}{Definition}

\newtheorem{remark}[theorem]{Remark}
\newtheorem{assumption}[theorem]{Hypothesis}

\numberwithin{equation}{section}
\DeclareMathOperator{\slim}{s-lim}

\author[J. Faupin]{J{\'e}r{\'e}my Faupin}
\address[J. Faupin]{Institut Elie Cartan de Lorraine \\
Universit{\'e} de Lorraine,
57045 Metz Cedex 1, France}
\email{jeremy.faupin@univ-lorraine.fr}

\author[F. Nicoleau]{François Nicoleau}
\address[F. Nicoleau]{Laboratoire de mathématiques Jean Leray, Université de Nantes,  2 rue de la Houssinière BP 92208 44322 Nantes Cedex 03, France}
\email{francois.nicoleau@univ-nantes.fr}


\begin{document}
\bibliographystyle{abbrv} \title[Scattering Matrices for dissipative operators]{Scattering matrices for dissipative quantum systems}

\begin{abstract}
We consider a quantum system $S$ interacting with another system $S'$ and susceptible of being absorbed by $S'$. The effective, dissipative dynamics of $S$ is supposed to be generated by an abstract pseudo-Hamiltonian of the form $H = H_0 + V - i C^* C$. The generator of the free dynamics, $H_0$, is self-adjoint, $V$ is symmetric and $C$ is bounded. We study the scattering theory for the pair of operators $(H,H_0)$. We establish a representation formula for the scattering matrices and identify a necessary and sufficient condition to their invertibility. This condition rests on a suitable notion of spectral singularity. Our main application is the nuclear optical model, where $H$ is a dissipative Schr{\"o}dinger operator and spectral singularities correspond to real resonances.
\end{abstract}

\maketitle

\section{Introduction}

When a quantum system $S$ interacts with another quantum system $S'$, part of its energy may be irreversibly transferred to $S'$. This phenomenon of irreversible loss of energy is usually called quantum dissipation. In particular, in the theory of open quantum systems, every small quantum system experiences energy dissipation due to interactions with its environment.

In this paper, we will be especially interested in situations where the interaction between $S$ and $S'$ may result in the absorption of $S$ by $S'$. A typical example is a neutron interacting with a nucleus: When a neutron is targeted onto a nucleus, it may either be elastically scattered off the nucleus, or be absorbed by it, forming, as suggested by Bohr \cite{Bo36_01}, a new system called a compound nucleus.

A particularly efficient model in Nuclear Physics allowing for the description of both elastic scattering and absorption in neutron-nucleus scattering is called the nuclear optical model. It was introduced by Feshbach, Porter and Weisskopf in \cite{FePoWe54_01}, as a simple model describing the scattering and compound nucleus formation by neutrons -- or protons -- impinging upon complex nuclei.

The starting point in \cite{FePoWe54_01} consists in considering an effective, pseudo-Hamiltonian for the neutron, in which the interaction with the nucleus is modeled by a complex potential with negative imaginary part. In other words, in suitable units, the dynamics of the neutron is generated by a pseudo-Hamiltonian of the form
\begin{equation}\label{eq:a1}
H = - \Delta + V(x) - i W(x) ,
\end{equation}
on $L^2( \mathbb{R}^3 )$, with $W \ge 0$. In mathematical terms, a pseudo-Hamiltonian corresponds to a maximal dissipative operator in a Hilbert space, i.e., a closed operator $H$ such that $H$ is dissipative (in the sense that $\mathrm{Im}( \langle u , H u \rangle ) \le 0$ for all $u$ in the domain of $H$) and $H$ has no proper dissipative extension. In particular, (see e.g. \cite{Da07_01} or \cite{EnNa20_01}), if $H$ is a pseudo-Hamiltonian then $-iH$ is the generator of a strongly continuous one-parameter semigroup of contractions $\{ e^{-itH} \}_{t\ge0}$.

In the nuclear optical model, if the neutron is initially in a normalized state $u_0 \in L^2(\mathbb{R}^3)$, its state at time $t\ge0$ is described by the unnormalized vector $u_t := e^{-itH}u_0$. The quantity
\begin{equation*}
p_{\mathrm{scatt}}(u_0):=\lim_{t\to\infty} \| u_t \|^2
\end{equation*}
gives the probability of elastic scattering, while
\begin{equation*}
p_{\mathrm{abs}}(u_0):=1-p_{\mathrm{scatt}}(u_0)
\end{equation*}
is the probability of  formation of a compound nucleus.

Concretely, given a specific situation, the form of the potentials in \eqref{eq:a1} is deduced from the scattering data gathered from experiments. Hence the central objects of study are the scattering cross-sections, or more generally the scattering matrices. Remarkably, theoretical computations based on pseudo-Hamiltonians of the form \eqref{eq:a1} are able to reproduce experimental data to a high degree of precision. Still, it is worth mentioning that many refinements of the original model introduced in \cite{FePoWe54_01} have subsequently been proposed, replacing for instance the potentials in \eqref{eq:a1} by non-local operators, or adding higher order terms such as spin-orbit interaction terms. We refer to, e.g., \cite{Ho71_01} or \cite{Fe92_01} for an overview of models considered in the literature.

The approach of Feshbach, Porter and Weisskopf may be generalized to any quantum system $S$ interacting with another quantum system $S'$ and susceptible to being absorbed by $S'$. If $\mathcal{H}$ is the complex Hilbert space corresponding to the pure states of $S$, the effective pseudo-Hamiltonian for $S$ is of the form
\begin{equation}
H := H_0 + V - i W ,
\end{equation}
on $\mathcal{H}$, where $H_0$ is the free Hamiltonian for $S$ and $V-iW$ represents the effective interaction due to the presence of $S'$, with $W$ non-negative.

Under suitable assumptions, one can define the scattering operator $S(H,H_0)$ for $H$ and $H_0$ by the formula
\begin{equation} \label{eq:def_scatt_op}
S(H,H_0) := ( W_+(H^*,H_0) )^* W_-(H,H_0) ,
\end{equation}
where $W_-(H,H_0)$ and $W_+(H^*,H_0)$ are the wave operators associated to $(H,H_0)$ and $(H^*,H_0)$, respectively. Here $H^*=H_0+V+iW$ stands for the adjoint of $H$. Precise definitions and conditions insuring the existence of $S(H,H_0)$ will be recalled in Section \ref{subsec:diss}.

If it exists, the scattering operator $S(H,H_0)$ commutes with $H_0$ and admits, as a consequence, a fiber decomposition of the form
\begin{equation}\label{eq:def_scatt_mat}
S(H,H_0) = \int^\oplus_\Lambda S(\lambda) d\lambda, \quad \text{ in } \quad \mathcal{H} = \int^\oplus_\Lambda \mathcal{H}(\lambda) d\lambda ,
\end{equation}
where the equalities should be understood as unitary equivalences (see \eqref{eq:def_scatt_mat2} for precise identities) and $\Lambda = \sigma(H_0)$ is the absolutely continuous spectrum of $H_0$. The operators $S(\lambda) : \mathcal{H}(\lambda) \to \mathcal{H}(\lambda)$, defined for almost every $\lambda \in \Lambda$, are called the scattering matrices.

Mathematical scattering theory for dissipative operators on Hilbert spaces has been considered by many authors. See e.g. \cite{Ma75_01,Da78_01,Da80_01,Ex85_01,Ka03_01,Ne85_01,Mo68_01,Si79_01,Ka65_01,WaZu14_01,FaFaFrSc17_01,FaFr18_01} and references therein. In particular, the existence and some properties of the wave and scattering operators have been established under various assumptions. However, to our knowledge, the scattering matrices have not been precisely analyzed in this context, yet. This constitutes the main purpose of the present paper. In unitary scattering theory, i.e., in scattering theory involving a pair of self-adjoint operators, the scattering matrices have been intensively studied. See e.g. \cite{ReSi79_01,Ya92_01,Ya10_01} for textbook presentations. Our argument will use and adapt to the dissipative setting several tools developed previously in the unitary scattering theory.

As far as the scattering matrices are concerned, an important difference between unitary and dissipative scattering theories is that, in the first case, scattering matrices are unitary, while in the second case they may not even be invertible. One of our main results will show that $S(\lambda)$ is invertible \emph{if and only if} $\lambda$ is not a spectral singularity of $H$, in a sense that will be made precise in the next section. For dissipative Schr{\"o}dinger operators of the form \eqref{eq:a1}, a spectral singularity corresponds to a real resonance.

Invertibility of the scattering matrix $S(\lambda)$ has the following physical interpretation: At energy $\lambda$, to any incoming scattering state corresponds a unique outgoing state, and vice versa. Moreover, as will be explained later on, the invertibility of the scattering matrices and operator are intimately related to the ``completeness'' of the theory. According to the latter, any state can be uniquely decomposed into a sum of a scattering state and a dissipative state (i.e. a state whose probability of absorption is equal to $1$). In this respect our results are to be compared to those recently proven in \cite{FaFr18_01}.

Our paper is organized as follows: In Section \ref{sec:results}, we state our main results. Section \ref{sec:prelim} contains several properties that follow directly from our assumptions. Section \ref{sec:proof} is devoted to the proofs of our theorems.

Throughout the paper, we will use the following notations: The resolvent of the self-adjoint operator $H_0$ will be denoted by $R_0(z) = (H_0-z)^{-1}$. If $X \subset \mathbb{R}$ is a borelian set, $E_0(X) = \mathds{1}_X(H_0)$ will stand for the spectral projection corresponding to $H_0$, on $X$. Likewise, for $H_V = H_0 + V$, we set $R_V(z) = (H_V-z)^{-1}$ and $E_V(X)=\mathds{1}_X(H_V)$. The resolvent of the dissipative operator $H$ will be denoted by $R(z)=(H-z)^{-1}$. Given a Hilbert space $\mathcal{H}$, its scalar product will be denoted by $\langle \cdot , \cdot \rangle_{\mathcal{H}}$, and the norm associated to it by $\| \cdot \|_{\mathcal{H}}$. The symbol $I$ will stand for the identity operator. The set of bounded operators from a Hilbert space $\mathcal{H}_1$ to another Hilbert space $\mathcal{H}_2$ will be denoted by $\mathcal{L}( \mathcal{H}_1 ; \mathcal{H}_2 )$, or $\mathcal{L}( \mathcal{H}_1 )$ if $\mathcal{H}_1=\mathcal{H}_2$. Likewise, the sets of compact operators will be denoted by  $\mathcal{L}^\infty( \mathcal{H}_1 ; \mathcal{H}_2 )$ or $\mathcal{L}^\infty( \mathcal{H}_1 )$.

\section{Hypotheses and statement of the main results}\label{sec:results}

\subsection{General setting.}

Let $\mathcal{H}$ be a complex separable Hilbert space corresponding to the pure states of a quantum system $S$. On this Hilbert space, we consider the operator
\begin{equation}\label{eq:exprH}
H := H_0 + V -  i C^* C ,
\end{equation}
where $H_0$, the free Hamiltonian for $S$, is self-adjoint and bounded from below, and $V-iC^*C$ represents the effective interaction between $S$ and another quantum system $S'$. Without loss of generality, we suppose that $H_0 \ge 0$. We assume in addition that $V$ is symmetric, $C \in \mathcal{L}( \mathcal{H} )$, and that $V$ and $C$ are relatively compact with respect to $H_0$. It follows that
\begin{equation*}
H_V := H_0 + V,
\end{equation*}
is self-adjoint on $\mathcal{H}$, with domain $\mathcal{D}( H_V ) = \mathcal{D}( H_0 )$, and that $H$ is a closed maximal dissipative operator with domain $\mathcal{D}( H ) = \mathcal{D}( H_0 )$. (Clearly, the operator $H$ is dissipative, since for all $u \in \mathcal{D}( H )$,
\begin{equation*}
\mathrm{Im} \langle u , H u \rangle_{\mathcal{H}}  = - \| C u \|_{\mathcal{H}}^2 \le 0.)
\end{equation*}
In particular, the spectrum of $H$ is contained in the lower half-plane, $\{ z \in \mathbb{C} , \mathrm{Im}( z ) \le 0 \}$ and  $ - i H$ is the generator of a strongly continuous one-parameter semigroup of contractions $\{ e^{ - i t H } \}_{ t \ge 0 }$ (see e.g. \cite{EnNa20_01} or \cite{Da07_01}). Actually, since $H$ is a perturbation of the self-adjoint operator $H_V$ by the bounded operator $- i C^* C $, $- i H$ generates a group $\{ e^{ - i t H }\}_{t \in \mathbb{R} }$ satisfying
\begin{equation*}
\big \| e^{ - i t H } \big \| \le 1 , \, \, t \ge 0 , \qquad \qquad \big \| e^{ - i t H } \big \| \le e^{ \| C^* C \| | t | } , \, \, t \le 0 ,
\end{equation*}
(see again \cite{EnNa20_01} or \cite{Da07_01}).

The essential spectrum of $H$ can be defined by  $\sigma_{ \mathrm{ess} }( H ) := \mathbb{C} \setminus \rho_{ \mathrm{ess} }( H )$, where (see e.g. \cite{EdEv87_01})
\begin{align*}
\rho_{ \mathrm{ess} }( H ) := \big \{ z \in \mathbb{C} , \, &\mathrm{Ran}( H - z ) \text{ is closed}, \,  \mathrm{dim} \, \mathrm{Ker}( H - z ) < \infty  \text{ or }  \mathrm{codim} \, \mathrm{Ran} ( H - z ) < \infty \big \} .
\end{align*}
Using that $V$ and $C$ are relatively $H_0$-compact, one then verifies (see \cite[Section IV.5.6]{Ka66_01} and \cite[proof of Proposition B.2]{Fr18_01}) that
\begin{equation*}
\sigma_{\mathrm{ess}}( H ) = \sigma_{ \mathrm{ess} }( H_V ) = \sigma_{ \mathrm{ess} }( H_0 )  .
\end{equation*}
Moreover the discrete spectrum $\sigma_{\mathrm{disc}}(H) := \sigma( H ) \setminus \sigma_{ \mathrm{ess} }( H )$ consists of isolated eigenvalues of finite multiplicities that can only accumulate at points of the essential spectrum.

Recall that the expression of the scattering operator $S(H,H_0)$, for the pair $(H,H_0)$, is given in \eqref{eq:def_scatt_op}. Assuming that the spectrum of $H_0$ is purely absolutely continuous -- this will be part of the content of Hypothesis  \ref{V-1} --, the wave operators $W_-(H,H_0)$ and $W_+(H^*,H_0)$ are defined by
\begin{equation*}
W_-( H , H_0 ) := \underset{t\to \infty }{\slim}  \, e^{ - i t H } e^{ i t H_0 }, \quad W_+( H^* , H_0 ) := \underset{t\to \infty }{\slim}  \, e^{ i t H^* } e^{ i t H_0 },
\end{equation*}
provided that the strong limits exist. Conditions insuring the existence of $W_-( H , H_0 )$ and $W_+( H^* , H_0 )$ will be given in Section \ref{subsec:hypoth}. Note that, if they exist, $W_-(H,H_0)$ and $W_+(H^*,H_0)$, and hence also $S(H,H_0)$, are contractions. Further basic properties of the wave and scattering operators will be recalled in Section \ref{sec:prelim}.

\subsection{Hypotheses.}\label{subsec:hypoth}

In this section, we state our main hypotheses. It will be explain in Section \ref{subsec:nuclear} that these hypotheses are satisfied by the nuclear optical model \eqref{eq:a1}, under suitable conditions on the potentials $V$ and $W$.

The first two hypotheses concern the spectra of the self-adjoint operators $H_0$ and $H_V$.


%
%
\begin{assumption}[Spectrum of $H_0$]\label{V-1}
The spectrum of $H_0$ is purely absolutely continuous and has a constant multiplicity (which may be infinite).
\end{assumption}
Assuming Hypothesis \ref{V-1}, the spectral theorem ensures that there exists a unitary mapping from $\mathcal{H}$ to a direct integral of Hilbert spaces,
\begin{equation}\label{eq:ST}
\mathcal{F}_0 : \mathcal{H} \rightarrow  \int^\oplus_{ \Lambda } \mathcal{H}( \lambda ) d \lambda , \quad \Lambda := \sigma(H_0),
\end{equation}
such that $\mathcal{F}_0  H_0 \mathcal{F}_0 ^*$ acts as  multiplication by $\lambda$ on each space $\mathcal{H}( \lambda )$. Moreover, since the absolutely continuous spectrum of $H_0$ has a constant multiplicity, say $k$, all spaces $\mathcal{H}( \lambda )$ can be identified with a fixed Hilbert space $\mathcal{M}$ and \eqref{eq:ST} reduces to
\begin{equation}\label{eq:ST1}
\mathcal{F}_0 : \mathcal{H} \rightarrow \int^\oplus_{ \Lambda } \mathcal{M} \, d \lambda = L^2 (\Lambda ;\mathcal{M}), \quad {\rm{dim}} \, \mathcal{M} =k,
\end{equation}
where $L^2 (\Lambda ;\mathcal{M})$ is  the space of square integrable functions from $\Lambda$ to $\mathcal{M}$, (see e.g. \cite[Chapter 0, Section 1.3]{Ya10_01}).  For $f,g \in \mathcal{F}_0^* ( \mathrm{C}_0^{\infty}(\Lambda, \mathcal{M}))$, we have that
\begin{equation}\label{eq:diagint}
\langle f,g\rangle_{\mathcal{H}}  =\int_{\Lambda} \big\langle \Gamma_0(\lambda)f , \Gamma_0(\lambda)g \big\rangle_{\mathcal{M}} \ d\lambda,
\end{equation}
where we recall that $\langle \cdot , \cdot \rangle_{\mathcal{H}}$ and  $\langle \cdot , \cdot \rangle_{\mathcal{M}}$ denote the scalar products on $\mathcal{H}$ and $\mathcal{M}$, respectively, and
where we have set
\begin{equation}\label{eq:def_Gamma0}
\Gamma_0(\lambda)f := \mathcal{F}_0f (\lambda) ,
\end{equation}
for a.e $\lambda \in \Lambda$. Since the scattering operator \eqref{eq:def_scatt_op} commutes with $H_0$, by \eqref{eq:ST1}, we have that
\begin{equation}\label{eq:def_scatt_mat2}
\mathcal{F}_0 S(H,H_0) \mathcal{F}_0^* = \int^\oplus_\Lambda S(\lambda) d\lambda ,
\end{equation}
where the scattering matrices identify with bounded operators
\begin{equation*}
S(\lambda) : \mathcal{M} \to \mathcal{M} ,
\end{equation*}
for a.e. $\lambda \in \Lambda$.

Note that in the nuclear optical model \eqref{eq:a1}, $H_0=-\Delta$ on  $\mathcal{H}= L^2(\mathbb{R}^3)$, $\mathcal{M} = L^2 (S^{2})$ where $S^2$ denotes the sphere in $\mathbb{R}^3$, and $\mathcal{F}_0$ is related to the usual Fourier transform $\mathcal{F}$ by the formula
\begin{equation}\label{eq:usualF0}
\mathcal{F}_0 f (\lambda)(\omega) = \lambda^{\frac{1}{4}} {\mathcal{F}{f}}(\sqrt{\lambda}\omega), \quad  \lambda >0, \quad \omega \in S^2.
\end{equation}


%
%
\begin{assumption}[Spectrum of $H_V$]\label{V-2}
The singular continuous spectrum of $H_V$ is empty, its pure point spectrum is finite dimensional and $H_V$ does not have embedded eigenvalues in its absolutely continuous spectrum.
\end{assumption}
%
%


Our third assumption concerns the structure of the symmetric operator $V$.

\begin{assumption}[Factorization of  $V$]\label{V-3}
There exist an auxiliary Hilbert space $\mathcal{G}$ and operators $G : \mathcal{H} \rightarrow \mathcal{G}$ and $K : \mathcal{G} \to \mathcal{G}$ such that
\begin{equation*}
V = G^* K G ,
\end{equation*}
where $G ( H_0^{\frac12}+1 )^{-1} \in \mathcal{L}( \mathcal{H} ; \mathcal{G} )$ and $K \in \mathcal{L}( \mathcal{G} )$. Moreover, for all $z \in \mathbb{C}$, $\mathrm{Im}(z) \neq 0$,
\begin{equation*}
G R_0(z) G^* \in \mathcal{L}^\infty(\mathcal{G}).
\end{equation*}
\end{assumption}
%
%


Recall that $E_0(X)$ denotes the spectral projection associated to $H_0$ on the borelian set $X \subset \mathbb{R}$. In our next assumption, we require that $G$ be strongly smooth with respect to $H_0$. This assumption will be important in order to define {\it{continuously}} the scattering matrices $S(\lambda)$ below.

\begin{assumption}[Regularity of $G$ with respect to $H_0$]\label{V-4}
The operator $G$ of Hypothesis \ref{V-3} is strongly $H_0$-smooth with exponent $s_0 \in (\frac12,1)$ on any compact set $X \Subset \Lambda$, i.e.
\begin{equation*}
\mathcal{F}_0  G_X^* : \mathcal{G} \rightarrow \mathrm{C}^{s_0}(X; \mathcal{M}) \text{ is continuous},
\end{equation*}
where $G_X :=G E_0(X)$ and $\mathrm{C}^{s_0}(X; \mathcal{M})$ denotes the set of H{\"o}lder continuous $\mathcal{M}$-valued functions of order $s_0$.
\end{assumption}

In other words, Hypothesis \ref{V-4} means that, for all $X \Subset \Lambda$,  there exists $\mathrm{c}_X>0$ such that, for all $f \in \mathcal{G}$ and $\lambda,  \mu \in X$,
\begin{equation}\label{eq:Ssmooth}
\big \| \mathcal{F}_0  G_X^* f(\lambda) \big \|_{\mathcal{M}} \leq \mathrm{c}_X  \|f\|_{\mathcal{G}}, \ \
\big \|Ê\mathcal{F}_0  G_X^* f(\lambda) - \mathcal{F}_0  G_X^* f(\mu) \big \|_{\mathcal{M}} \leq \mathrm{c}_X | \lambda - \mu |^{s_0} \|f\|_{\mathcal{G}}.
\end{equation}
We suppose that $s_0>1/2$ in order to insure that the only possible ``strictly embedded singularities'' of $H_V$ are eigenvalues, see the proof of Proposition \ref{prop:standard_csq} for more details.

It will be recalled in the next section that Hypotheses \ref{V-1}--\ref{V-4} imply the existence of the unitary wave operators
\begin{equation*}
W_\pm( H_V , H_0 ) := \underset{t\to \pm \infty }{\slim} e^{ i t H_V } e^{ - i t H_0 } .
\end{equation*}
The scattering operator for the pair $(H_V, H_0)$ is defined by
\begin{equation}\label{eq:scatop}
S(H_V, H_0) := W_+^*( H_V , H_0 )W_-( H_V , H_0 ).
\end{equation}
It is unitary on $\mathcal{H}$ and commutes with $H_0$. The scattering matrices $S_V(\lambda) : \mathcal{M} \to \mathcal{M}$ are defined as in \eqref{eq:def_scatt_mat}, using the spectral decomposition operator $\mathcal{F}_0$ in \eqref{eq:ST}--\eqref{eq:ST1}, i.e.
\begin{equation}
\mathcal{F}_0 S( H_V , H_0 ) \mathcal{F}_0^* = \int^\oplus_{ \Lambda } S_V( \lambda ) d \lambda. \label{eq:def_S(HV,H0)}
\end{equation}
The operator $S_V( \lambda ) \in \mathcal{L}(\mathcal{M})$ is unitary for a.e. $\lambda \in \Lambda$, $S_V(\lambda)-I$ is compact, and it follows from Hypotheses \ref{V-1}--\ref{V-4} that the map $\mathring{\Lambda} \ni \lambda \mapsto S_V(\lambda) \in \mathcal{L}(\mathcal{M})$ is continuous (see e.g. \cite{Ya92_01,Ya10_01}; see also Proposition \ref{prop:standard_csq} below). Formally, $S_V(\lambda)$ are given by the well-known formula
\begin{equation}\label{eq:Kurodabis}
 S_V(\lambda)  = I -2i\pi \Gamma_0(\lambda) (V-VR_V(\lambda+i0)V)\Gamma_0^*(\lambda).
\end{equation}
See again Proposition \ref{prop:standard_csq} for a rigorous expression and justifications of notations.

We set
\begin{equation}\label{eq:def_wave}
\mathcal{F}_\pm := \mathcal{F}_0 W^*_\pm( H_V , H_0 )  : \mathcal{H} \rightarrow L^2 (\Lambda ;\mathcal{M}) ,
\end{equation}
and we note that $\mathcal{F}_\pm$ can be explicitly written as
\begin{equation}\label{eq:diagonalisation_0}
\mathcal{F}_{\pm}f (\lambda) = \Gamma_{\pm}(\lambda) f ,
\end{equation}
for $f \in \mathcal{H}$ and a.e. $\lambda \in \Lambda$, see \eqref{eq:gammapm} in Section \ref{sec:prelim}. Equation \eqref{eq:diagonalisation_0} should be compared to \eqref{eq:def_Gamma0}.


Recall that $E_V(X)$ denotes the spectral projection associated to $H_V$ on the interval $X$. In our next assumption, we require that the operator $C$ be strongly $H_V$-smooth. Here it should be noticed that Hypothesis \ref{V-2} together with the fact that $\sigma_{\mathrm{ess}}(H_V)=\sigma(H_0)$ yield
\begin{equation*}
\sigma_{\mathrm{ac}}(H_V) = \sigma(H_0)= \Lambda.
\end{equation*}
\begin{assumption}[Regularity of $C$ with respect to $H_V$]\label{V-5}
The operator $C$ is strongly $H_V$-smooth with exponent $s \in ( 0 , 1)$ on any compact set $X \Subset \Lambda$, i.e.
\begin{equation}
\mathcal{F}_{\pm}  C_X^* : \mathcal{H} \to \mathrm{C}^s( \Lambda ; \mathcal{H}) \text{ is continuous}, \label{eq:HypV-5}
\end{equation}
where $C_X :=C E_{V}( X )$. Moreover, $C R_V(z) C^*$ is compact for all $z\in\mathbb{C}$,  $\mathrm{Im}(z)\neq0$, and the map
\begin{equation}
\mathring{\Lambda} \in \lambda \mapsto C \big ( R_V( \lambda + i 0 ) - R_V( \lambda - i 0 ) \big ) C^* \in \mathcal{L}^\infty( \mathcal{H} ) , \label{eq:LAP_bounded}
\end{equation}
is bounded.
\end{assumption}
Similarly as for Hypothesis \ref{V-4}, \eqref{eq:HypV-5} means that, for all $X \Subset \Lambda$, there exists $\mathrm{c}_X>0$ such that, for all $f \in \mathcal{H}$ and $\lambda,  \mu \in \Lambda$,
\begin{equation*}
\big \| \mathcal{F}_\pm  C_X^* f(\lambda) \big \|_{\mathcal{M}} \leq \mathrm{c}_X  \|f\|_{\mathcal{H}}, \ \
\big \|Ê\mathcal{F}_\pm  C_X^* f(\lambda) - \mathcal{F}_\pm  C_X^* f(\mu) \big \|_{\mathcal{M}} \leq \mathrm{c}_X | \lambda - \mu |^{s} \|f\|_{\mathcal{H}}.
\end{equation*}
Moreover, the notations used in \eqref{eq:LAP_bounded} are justified by the fact  that the limits
\begin{equation*}
C R_V( \lambda \pm i 0 ) C^* := \lim_{\varepsilon \downarrow 0} C R_V( \lambda \pm i \varepsilon ) C^*,
\end{equation*}
 exist for all $\lambda \in \mathring{\Lambda}$, assuming \eqref{eq:HypV-5}. This will be explained in Proposition \ref{prop:standard_csq_2}. More precisely, \eqref{eq:HypV-5} implies that the map in $\eqref{eq:LAP_bounded}$ is locally H{\"o}lder continuous. Hence, roughly speaking, the condition that the map in \eqref{eq:LAP_bounded} is bounded means that, for all $\lambda_0 \in \Lambda \setminus \mathring{\Lambda}$,
 \begin{equation*}
 C \big ( R_V( \lambda + i 0 ) - R_V( \lambda - i 0 ) \big ) C^*
 \end{equation*}
 does not blow up as $\lambda \to \lambda_0$.

\subsection{Main results.}

Our first result establishes, in the dissipative setting, a representation formula for the scattering matrices $S(\lambda) : \mathcal{M} \to \mathcal{M}$. This generalizes usual Kuroda's representation formula \cite[Theorem 5.4.4']{Ya92_01} established in unitary scattering theory.  Recall that the unitary scattering matrices $S_V(\lambda)$ for the pair of self-adjoint operators $(H_V,H_0)$ are defined in \eqref{eq:def_S(HV,H0)}--\eqref{eq:Kurodabis} (see also \eqref{eq:Kuroda} below).

\begin{theorem}\label{thm:RF}
Suppose that Hypotheses \ref{V-1}--\ref{V-5} hold. The scattering matrices $S(\lambda)\in \mathcal{L}(\mathcal{M})$ for the pair of operators $(H, H_0)$ depend continuously on $\lambda \in \mathring{\Lambda}$
and are given by
\begin{align*}
S(\lambda) &= \Big( I -2\pi \Gamma_{+}(\lambda) C^* (I+iCR(\lambda+i0) C^*) C \Gamma_{+}(\lambda)^* \Big) S_V(\lambda)  ,\\
           &=  \Big( I -2\pi \Gamma_{+}(\lambda) C^* (I-iCR_V(\lambda+i0)C^*)^{-1}  C \Gamma_{+}(\lambda)^* \Big) S_V(\lambda) ,
\end{align*}
for all $\lambda \in \mathring{\Lambda}$. Moreover, for all $\lambda \in \mathring{\Lambda}$, $S(\lambda)$ is a contraction and $S(\lambda)-I$ is compact. In particular, if in addition $\mathrm{dim} \, \mathcal{M} = + \infty$, then $\| S(\lambda) \|Ê= 1$ for all $\lambda \in \mathring{\Lambda}$.
\end{theorem}
\begin{remark}\label{rk:kuroda}
\begin{enumerate}[label=\rm{(\roman*)},leftmargin=*]
\item In the particular case where $V=0$, assuming that Hypotheses \ref{V-1} and \ref{V-5} with $V=0$ hold, the previous result reduces to the following statement: Let $\tilde{S}(\lambda) \in \mathcal{L}(\mathcal{M})$ be the scattering matrices for the pair of operators $(H_0 -iC^* C, H_0)$. They  depend continuously on $\lambda\in\mathring{\Lambda}$ and  we have that
\begin{align}
\tilde{S}(\lambda) &= I -2\pi \Gamma_0(\lambda) C^* (I+iCR(\lambda+i0) C^*) C \Gamma_0(\lambda)^* , \label{eq:S0(lambda)1} \\
             &= I -2\pi \Gamma_0(\lambda) C^* (I-iCR_0(\lambda+i0)C^*)^{-1}  C \Gamma_0(\lambda)^*, \label{eq:S0(lambda)2}
\end{align}
for all $\lambda \in \mathring{\Lambda}$. Moreover, $\tilde{S}(\lambda)$ is a contraction (of norm $1$ if $\mathcal{M}$ is infinite dimensional) and $\tilde{S}(\lambda)-I$ is compact.
\item If the operator $H_V$ is ``diagonalized'' using $\Gamma_- (\lambda)$ instead of $\Gamma_+(\lambda)$ then, under the same hypotheses, the following representation formulae hold:
 \begin{align*}
S(\lambda) &=  S_V(\lambda) \Big( I -2\pi \Gamma_{-}(\lambda) C^* (I+iCR(\lambda+i0) C^*) C \Gamma_{-}(\lambda)^* \Big), \\
           &=  S_V(\lambda) \Big( I -2\pi \Gamma_{-}(\lambda) C^* (I-iCR_V(\lambda+i0)C^*)^{-1}  C \Gamma_{-}(\lambda)^* \Big).
\end{align*}
\item Under the assumptions of Theorem \ref{thm:RF}, and assuming in addition that $\mathrm{dim} \, \mathcal{M} = + \infty$, the property $\| S(\lambda) \|Ê= 1$ for all $\lambda \in \mathring{\Lambda}$ implies that
\begin{equation*}
\| S( H , H_0 ) \|Ê= 1 ,
\end{equation*}
(see, e.g., \cite[Theorem XIII.83]{ReSi78_01}).
\end{enumerate}
\end{remark}
Our next concern is to find a necessary and sufficient condition for the invertibility of $S(\lambda)$. We use the notion of spectral singularity introduced in \cite{FaFr18_01}. This notion is closely related to that considered in \cite{Du58_01,DuSc71_01,Sc60_01} within the theory of ``spectral operators''.
\begin{definition}\label{def:spec-sing}
\begin{enumerate}[label=\rm{(\roman*)},leftmargin=*]
\item Let $\lambda \in \mathring{\Lambda}$. We say that $\lambda$ is a regular spectral point of $H$ if there exists a compact interval $K_\lambda \subset \mathbb{R}$ whose interior contains $\lambda$, such that $K_\lambda$ does not contain any accumulation point of eigenvalues of $H$, and such that the limits
\begin{equation*}
C R ( \mu - i 0 ) C^* := \lim_{ \varepsilon \downarrow 0 } C R ( \mu - i \varepsilon ) C^*
\end{equation*}
exist uniformly in $\mu \in K_\lambda$ in the norm topology of $\mathcal{L}( \mathcal{H} )$. If $\lambda$ is not a regular spectral point of $H$, we say that $\lambda$ is a spectral singularity of $H$.
\item Let $\lambda \in \Lambda \setminus \mathring{\Lambda}$. We say that $\lambda$ is a regular spectral point of $H$ if there exists a compact interval $K_\lambda \subset \mathbb{R}$ whose interior contains $\lambda$, such that all $\mu \in K_\lambda \cap \mathring{\Lambda}$ are regular in the sense of (i), and such that the map
\begin{equation}\label{eq:bound_regular_threshold}
K_\lambda \cap \mathring{\Lambda} \ni \mu \mapsto C R ( \mu - i 0 ) C^* \in \mathcal{L}(\mathcal{H})
\end{equation}
is bounded.
\item If $\Lambda$ is right-unbounded, we say that $+\infty$ is regular if there exists $m>0$ such that all $\mu \in [m,\infty) \cap \mathring{\Lambda}$ are regular in the sense of (i), and such that the map
\begin{equation*}
[m,\infty) \cap \mathring{\Lambda} \ni \mu \mapsto C R ( \mu - i 0 ) C^* \in \mathcal{L}(\mathcal{H})
\end{equation*}
is bounded.
\end{enumerate}
\end{definition}
\begin{remark}\label{rk:defspec-sing}
\begin{enumerate}[label=\rm{(\roman*)},leftmargin=*]
\item  Our definitions of spectral regularity have a local meaning. For $\lambda \in \mathring{\Lambda}$, a weaker natural definition of a regular spectral point would be to require that $\lambda$ is not an accumulation point of eigenvalues of $H$ located in $\lambda-i(0,\infty)$ and that the limit
\begin{equation*}
C R ( \lambda - i 0 ) C^* := \lim_{ \varepsilon \downarrow 0 } C R ( \lambda - i \varepsilon ) C^*
\end{equation*}
exists in the norm topology of $\mathcal{L}( \mathcal{H} )$. It will be shown in Lemma \ref{lem:criterion} that, under our assumptions, this definition is actually \emph{equivalent} to Definition \ref{def:spec-sing} (i).
\item It will be shown in Section \ref{subsec:thminvert} that, under our assumptions, the set of spectral singularities in $\mathring{\Lambda}$ is a closed set of Lebesgue measure $0$.
\end{enumerate}
\end{remark}
The notion of spectral singularities in the nuclear optical model \eqref{eq:a1} and its relation with the notion of resonances will be discussed in Section \ref{subsec:nuclear} (see also \cite[Section 6]{FaFr18_01}).

Our main result is stated in the following theorem.
\begin{theorem}\label{thm:invert}
Suppose that Hypotheses \ref{V-1}--\ref{V-5} hold and let $\lambda \in \mathring{\Lambda}$. Then $S(\lambda)$ is invertible in $\mathcal{L}(\mathcal{M})$ if and only if
$\lambda$ is a regular spectral point of $H$ in the sense of Definition \ref{def:spec-sing}. In this case, its inverse is given by
\begin{align}
S(\lambda)^{-1} &= S_V(\lambda)^{-1} \Big( I + 2\pi \Gamma_{+}(\lambda) C^* (I+iCR(\lambda-i0) C^*) C \Gamma_{+}(\lambda)^* \Big)  , \label{eq:S^-1} \\
           &= S_V(\lambda)^{-1} \Big( I +2\pi \Gamma_{+}(\lambda) C^* (I-iCR_V(\lambda-i0)C^*)^{-1}  C \Gamma_{+}(\lambda)^* \Big)  . \label{eq:S^-1bis}
\end{align}
\end{theorem}
Theorem \ref{thm:invert} has the following consequence for the scattering operator $S(H,H_0)$ and the wave operator $W_-(H,H_0)$. We recall that, in the context of dissipative scattering theory, the space of \emph{bound states} can be defined by
\begin{equation}\label{eq:defHb}
\mathcal{H}_{ \mathrm{b} }( H ) := \mathrm{Span} \big \{ u \in \mathcal{D}( H ) , \, \exists \lambda \in \mathbb{R} , \, H u = \lambda u \big \} .
\end{equation}
It coincides with the space of bound states for $H^*$ (see \cite[Lemma 3.1]{FaFr18_01} or \cite[Lemma 1]{Da80_01}). The \emph{dissipative spaces}, for $H$ and $H^*$, respectively, are defined by
\begin{align}
& \mathcal{H}_{ \mathrm{d} }( H ) := \big \{ u \in \mathcal{H} , \lim_{ t \to \infty } \|e^{ - i t H }u \|_{\mathcal{H}} = 0 \big \} , \label{eq:defHd} \\
& \mathcal{H}_{ \mathrm{d} }( H^* ) := \big \{ u \in \mathcal{H} , \lim_{ t \to \infty } \|e^{ i t H^* }u \|_{\mathcal{H}} = 0 \big \}. \label{eq:defH*d}
\end{align}
\begin{theorem}\label{thm:invert_S}
Suppose that Hypotheses \ref{V-1}--\ref{V-5} hold. Assume that $\Lambda \setminus \mathring{\Lambda}$ is finite and that all $\lambda \in \Lambda \setminus \mathring{\Lambda}$ are regular in the sense of Definition \ref{def:spec-sing}. If $\Lambda$ is right-unbounded, assume in addition  that $+\infty$ is regular.
Then $S(H,H_0)$ is invertible in $\mathcal{L}(\mathcal{H})$ if and only if $H$ does not have spectral singularities in $\mathring{\Lambda}$. In this case, in particular, the range of the wave operator $W_-( H , H_0 )$ is closed and given by
\begin{align}\label{eq:RanW-closed}
&Ê\mathrm{Ran}( W_-( H , H_0 ) ) = \big ( \mathcal{H}_{ \mathrm{b} }( H ) \oplus \mathcal{H}_{ \mathrm{d} }( H^* ) \big )^\perp.
\end{align}
\end{theorem}
\begin{remark}
The assumptions that $\Lambda \setminus \mathring{\Lambda}$ is finite and that all $\lambda \in \Lambda \setminus \mathring{\Lambda}$ are regular can be replaced by the condition that  the supremum, over $\lambda \in \Lambda \setminus \mathring{\Lambda}$, of the $L^\infty$-norms of the maps in \eqref{eq:bound_regular_threshold} is finite.
\end{remark}
A related result -- the scattering operator $S(H,H_0)$ is invertible if and only if $H$ does not have spectral singularities -- has been proven recently in \cite{FaFr18_01}, under different assumptions and following a different approach. In particular, it is assumed in \cite{FaFr18_01} that $H$ has only finitely many eigenvalues, which we do not impose here. In view of applications to Schr{\"o}dinger operators of the form \eqref{eq:a1}, this constitutes a substantial improvement since the potentials $V$ and $W$ are supposed to be compactly supported (or exponentially decaying) in \cite{FaFr18_01}, while we only need to impose a polynomial decay here. See Section \ref{subsec:nuclear} for details. Besides, in \cite{FaFr18_01}, a further abstract condition is needed in order to prove that if $H$ has a spectral singularity then $S(H,H_0)$ is not invertible (see \cite[Theorem 2.10]{FaFr18_01}). This condition is relaxed in the present paper. On the other hand, as for \eqref{eq:RanW-closed}, the results established in \cite{FaFr18_01} are significantly more precise since they show, in addition, that $\mathcal{H}_{ \mathrm{d} }( H^* )$ identifies to the vector space spanned by the generalized eigenvectors of $H^*$ corresponding to eigenvalues with positive imaginary parts (see \cite[Theorem 2.8]{FaFr18_01} or Theorem \ref{thm:mainFF} below for a precise statement). It would be interesting to prove that this result remains true under our assumptions.

\subsection{Application to the nuclear optical model}\label{subsec:nuclear}

We recall that for the nuclear optical model, the hamiltonian is given by $H = H_0 + V - i W $  on $\mathcal{H} = L^2(\mathbb{R}^3)$, where $H_0=-\Delta$ and $V$, $W$ are measurable bounded potentials, with $W \ge 0$. We write $W=C^*C$ with $C=C^*=\sqrt{W}$. Moreover, $\mathcal{M} = L^2 (S^{2})$  and $\mathcal{F}_0$ is related to the usual Fourier transform $\mathcal{F}$ by the formula \eqref{eq:usualF0}.

In this section we provide conditions on $V$ and $W$ insuring that Hypotheses \ref{V-1}--\ref{V-5} are satisfied.

\vspace{0,2cm}
\noindent \textbf{Verification of Hypothesis \ref{V-1}.} Clearly, the operator $-\Delta$ on $L^2( \mathbb{R}^3 )$ has a purely absolutely continuous spectrum equal to $[0,\infty)$, with infinite, constant multiplicity.

\vspace{0,2cm}
\noindent \textbf{Verification of Hypothesis \ref{V-2}.} For Schr{\"o}dinger operators $H_V = - \Delta + V(x)$ acting on $L^2(\mathbb{R}^3)$, assuming that $V$ is bounded and satisfies $V(x) = \mathcal{O}( \langle x \rangle^{-\rho} )$ in a neighborhood of $\infty$, with $\rho>1$, it is known that $H_V$ has no singular spectrum and no embedded eigenvalues except, perhaps, at $0$ (see e.g. \cite[Theorem XIII.33 and Theorem XIII.58]{ReSi78_01}). Let us mention that the absence of embedded positive eigenvalue also holds if $V \in L^2(\mathbb{R}^3)$ (see \cite{KoTa06_01} and references therein). Moreover, if $V(x) = \mathcal{O}( \langle x \rangle^{-\rho} )$ near $\infty$ with $\rho>1$, then the negative part of the potential $V_- = \sup (-V,0)$ belongs to $L^{3/2}(\mathbb{R}^3)$ and hence the Cwikel-Lieb-Rosembljum bound (see \cite[Theorem XIII.12]{ReSi78_01}) ensures that $H_V$ has at most  finitely many eigenvalues counting multiplicities.

\vspace{0,2cm}
\noindent \textbf{Verification of Hypothesis \ref{V-3}.} In the nuclear optical model \eqref{eq:a1}, $\mathcal{H}= \mathcal{G} = L^2 (\mathbb{R}^3)$. If $V$ is the operator of multiplication by a bounded potential satisfying
$V(x) = \mathcal{O}( \langle x \rangle^{-2\rho})$ with $\rho>0$, we can take $G= \sqrt{|V|}$ and $K =\mathrm{sgn}(V)$.
Another possible choice is $G = \langle x \rangle^{-\alpha}$, with $0 < \alpha \leq \rho$, and $K = \langle x \rangle^{2\alpha} V$.

\vspace{0,2cm}
\noindent \textbf{Verification of Hypothesis \ref{V-4}.} Hypothesis \ref{V-4} is satisfied in the nuclear optical model \eqref{eq:a1} if $V$ is bounded and satisfies $V = \mathcal{O}( \langle x \rangle^{-2\alpha})$ near $\infty$, with $\alpha > \frac{1}{2}$.
Indeed, if we take $G = \langle x \rangle^{-\alpha}$, and $K = \langle x \rangle^{2\alpha} V$, then $G$ is strongly $H_0$-smooth with exponent belonging to $( \frac{1}{2},1)$ (see \cite[Proposition 1.6.1]{Ya10_01}).

\vspace{0,2cm}
\noindent \textbf{Verification of Hypothesis \ref{V-5}.} In the nuclear optical model \eqref{eq:a1}, it follows from \cite[Proposition 1.6.1]{Ya10_01} that, for any $s > \frac{1}{2}$,  the map $\mathring{\Lambda} \ni \lambda \to \Gamma_0(\lambda) \langle x \rangle^{-s} \in \mathcal{L}(L^2(\mathbb{R}^3), L^2(S^2))$ is Hölder continuous. Moreover, if $V \in \mathrm{C}^2(\mathbb{R}^3) $ and satisfies, in a neighborhood of infinity, $\partial_x^{\alpha} V(x)=\mathcal{O}( \langle x \rangle^{-\rho-|\alpha|} )$, $|\alpha| \leq 2$, for some $\rho>0$, then the limiting absorption principle (see \cite{Mou81_01,JMP84_01,BoGe96_01}) asserts  that,  for $s > \frac{1}{2}$,  $\| \langle x \rangle^{-s} (R_V(z)-R_V(z') ) \langle x \rangle^{-s} \| \leq C \,  |z-z'|^{s-1/2}$ for all $z,z' \in \{ z \in \mathbb{C} , \, \pm \mathrm{Im}(z) \ge 0 \}$. Thus, if in addition $W(x) = \mathcal{O}( \langle x \rangle^{-\delta} )$ near $\infty$ with $\delta>1$, we see that $\lambda \to \Gamma_{\pm}(\lambda) C^*$ (with $C=\sqrt{W}$) is Hölder continuous. In other words, $C$ is strongly $H_V$-smooth.

To insure that \eqref{eq:LAP_bounded} holds for the nuclear optical model, it remains to verify that $C R_V(\lambda\pm i0) C^*$ is bounded for $\lambda$ in a neighborhood of $0$ and $\infty$. Near $\infty$, using the explicit form of the resolvent of $-\Delta$ and a Neumann series argument, it is well-known (see for instance \cite[Lemma 2.1]{Je80_01}) that, for any $V$ bounded, $V(x) = \mathcal{O}( \langle x \rangle^{-\rho} )$ near $\infty$ with $\rho>1$, the following estimate holds:
\begin{equation}
\| \langle x \rangle^{-s} R_V (\lambda \pm i0) \langle x \rangle^{-s} \| = \mathcal{O}(\lambda^{-\frac{1}{2}}) ,\quad  \lambda \rightarrow \infty ,
\end{equation}
provided that $s > \frac{1}{2}$. In particular, if in addition $W(x) = \mathcal{O}( \langle x \rangle^{-\delta} )$ near $\infty$ with $\delta>1$, then
\begin{equation}\label{eq:art_1}
\lambda \mapsto C \big ( R_V( \lambda + i 0 ) - R_V( \lambda - i 0 ) \big ) C^*
\end{equation}
is indeed bounded in a neighborhood of $\infty$. Note that if $V$ is supposed to be of class $\mathrm{C}^2$, the short-range condition $V(x) = \mathcal{O}( \langle x \rangle^{-\rho} )$ near $\infty$ with $\rho>1$ can be replaced by the long-range condition $V(x) = \mathcal{O}( \langle x \rangle^{-\rho} )$ near $\infty$ with $\rho>0$ (see e.g. \cite{Ro90_01}).

To control \eqref{eq:art_1} in a neighborhood of $\lambda=0$, we need to make additional assumptions. We suppose that $V(x) = \mathcal{O}( \langle x \rangle^{-\rho} )$ with $\rho>3$, and that $0$ is neither an eigenvalue nor a resonance of $H_V$. These two conditions then imply that, for any $s>\frac{1}{2}$, $\langle x \rangle^{-s} R_V(\lambda \pm i0) \langle x \rangle^{-s}$ is bounded near $\lambda=0$ (see \cite{JeKa79_01,Sc07_01,ErSc04_01,JeNe01_01} for details), and hence \eqref{eq:art_1} is bounded near $0$ provided that $W(x) = \mathcal{O}( \langle x \rangle^{-\delta} )$ near $\infty$ with $\delta>1$. Let us mention that if one imposes a sign condition on $V$ near infinity, together with a positive virial condition, also near infinity, then the polynomial decay $V(x) = \mathcal{O}( \langle x \rangle^{-\rho} )$ with $\rho > 3$ can be weakened to $\rho > 0$; see \cite{FoSk04_01}.

\vspace{0,2cm}
\noindent \textbf{Spectral singularities in the nuclear optical model.}
For dissipative Schr{\"o}dinger operators of the form \eqref{eq:a1}, if $V : \mathbb{R}^3 \to \mathbb{R}$ and $C : \mathbb{R}^3 \to \mathbb{C}$ are bounded, decaying potentials, we have that $\Lambda = [0,\infty)$ and a spectral singularity $\lambda \in (0,\infty)$ corresponds to the real resonance $-\lambda^{1/2}$, (see \cite[Section 6]{FaFr18_01} for details). Recall that if $V$ and $C$ are supposed to be compactly supported, a resonance may be defined as a pole of the map
\begin{equation*}
\mathbb{C} \ni z \mapsto ( H - z^2 )^{-1} : L^2_{ \mathrm{c} }( \mathbb{R}^3 ) \to L^2_{ \mathrm{loc} }( \mathbb{R}^3 ) ,
\end{equation*}
(given as the meromorphic extension of the meromorphic map originally defined on $\{ z \in \mathbb{C} , \mathrm{Im}( z ) > 0 \}$), where $L^2_{ \mathrm{c}Ê}( \mathbb{R}^3 )= \{ u \in L^2( \mathbb{R}^3 ) , u \text{ is compactly supported} \}$ and $L^2_{Ê\mathrm{loc} }( \mathbb{R}^3 ) = \{ u : \mathbb{R}^3 \to \mathbb{C}, u \in L^2( K ) \text{ for all compact set } K \subset \mathbb{R}^3 \}$. If $V$ and $W=C^*C$ are supposed to satisfy $V(x) = \mathcal{O}( \langle x \rangle^{-\rho} )$ and $W(x) = \mathcal{O}( \langle x \rangle^{-\delta} )$ near $\infty$, with $\rho,\delta >2$, then, more generally, $\pm \lambda^{1/2}$ (with $\lambda>0$) may be called a resonance of $H=-\Delta+V-iW$ if the equation $(H-\lambda)u=0$ admits a solution $u \in H^2_{\mathrm{loc}}(\mathbb{R}^3) \setminus L^2( \mathbb{R}^3 )$ satisfying the Sommerfeld radiation condition
\begin{equation*}
u(x)=|x|^{\frac32 - 1 } e^{\pm i \lambda^{\frac12} | x |Ê} \Big ( a ( \frac{x}{|x|} ) + o(1) \Big ) , \quad |x| \to \infty ,
\end{equation*}
with $a \in L^2( S^2 )$, $a\neq 0$.

As for the regularity of the threshold $0$, and of $+\infty$, we can rely on the following results. Assuming that $V(x) = \mathcal{O}( \langle x \rangle^{-\rho} )$ and
$W(x) = \mathcal{O}( \langle x \rangle^{-\delta} )$ near $\infty$, with $\rho,\delta >2$, $W>0$ on a nontrivial set,  it is proven in \cite[Theorem 1.1]{Wa11_01} that $0$ is a regular spectral point of $H=-\Delta+V-iW$. As for the regularity of $\infty$, one can use as before a Neumann series argument together with the fact that
\begin{equation*}
\rho_1 ( - \Delta - (\lambda - i 0 ) )^{-1} \rho_2 = \mathcal{O}( \lambda^{-\frac12} ),
\end{equation*}
as $\lambda \to \infty$, for any $\rho_1,\rho_2:\mathbb{R}^3 \to \mathbb{R}$ such that $| \rho_i(x) | \le C \langle x \rangle^{-\gamma}$, $\gamma > 1/2$ (see for instance \cite[Lemma 2.1]{Je80_01}).

\vspace{0,2cm}

Now, summarizing the above discussion, we obtain the following result.
\begin{theorem}\label{thm:optical1}
Let $V,W \in L^\infty( \mathbb{R}^3 ; \mathbb{R} )$. Assume that $V \in \mathrm{C}^2(\mathbb{R}^3)$ and that $V$ satisfies, for $|\alpha|\leq 2$,  $\partial_x^{\alpha} V(x) = \mathcal{O}( \langle x \rangle^{-\rho-|\alpha|} )$ near $\infty$ with $\rho>3$. Suppose that $W(x) \ge 0$, $W(x)>0$ on a non-trivial open set and $W(x) = \mathcal{O}( \langle x \rangle^{-\delta} )$ near $\infty$ with $\delta>2$. Assume in addition that $0$ is neither an eigenvalue nor a resonance of $H_V=-\Delta+V(x)$.

For all $\lambda>0$, the scattering matrices corresponding to $H=-\Delta+V(x)-iW(x)$ and $H_0=-\Delta$ are given by
\begin{align*}
S(\lambda) &= \Big( I -2\pi \Gamma_{+}(\lambda) (W-iWR(\lambda+i0)W) \Gamma_{+}(\lambda)^* \Big) S_V(\lambda)  ,\\
           &=  \Big( I -2\pi \Gamma_{+}(\lambda) {\sqrt{W}} (I-i{\sqrt{W}}R_V(\lambda+i0){\sqrt{W}})^{-1}  {\sqrt{W}} \Gamma_{+}(\lambda)^* \Big) S_V(\lambda) ,
\end{align*}
where $S_V(\lambda)$ are the scattering matrices corresponding to $H_V$ and $H_0$.

Moreover, for all $\lambda>0$, the operator $S(\lambda) \in \mathcal{L}( L^2( S^2 ) )$ is invertible if and only if $\lambda$ is not a spectral singularity of $H$.

Finally, the scattering operator $S(H,H_0)$ is invertible in $\mathcal{L}( L^2( \mathbb{R}^3 ) )$ if and only if $H$ does not have spectral singularities in $(0,\infty)$, and, in this case, we have that
\begin{align*}
Ê\mathrm{Ran}( W_-( H , H_0 ) ) = \mathcal{H}_{ \mathrm{d} }( H^* )^\perp.
\end{align*}
\end{theorem}
As for the last statement of Theorem \ref{thm:optical1}, which should be compared to \eqref{eq:RanW-closed}, we emphasize that $\mathcal{H}_{\mathrm{b}}(H) = \{ 0 \}$ for dissipative Schr{\"o}dinger operators. This follows from the unique continuation principle (see, e.g., \cite[Theorem XIII.63]{ReSi78_01}) and the assumption that $W(x)>0$ on some non-trivial open set.

We mention that, in the particular case where $V$ and $W$ are compactly supported, the map $\lambda \mapsto S(\lambda^2)$ has a meromorphic continuation to the entire complex plane (see e.g. \cite{Ya92_01}) and $H$ has finitely many spectral singularities (see e.g. \cite{DyZw17_01}).  Furthermore, using the representation formula of Theorem \ref{thm:optical1}, it is possible to verify that, in a suitable sense, any $\lambda > 0$ is generically a regular spectral point of $H$. On the other hand, for any $\lambda>0$, one can construct compactly supported potentials $V$ and $W$ such that $\lambda$ is a spectral singularity of $H$ (see \cite{Wa12_01}). If the $L^\infty$-norms of $V$ and $W$ are small enough, it is known that $H$ is similar to $H_0$ and hence, in particular, $H$ does not have spectral singularities (see \cite{Ka65_01,FaFaFrSc17_01}). It would be interesting to find more general conditions on $V$ and $W$ insuring the absence of spectral singularities for $H$.

\section{Basic properties}\label{sec:prelim}

In this section we recall various properties that follow from our hypotheses. Subsection \ref{subsec:smooth} is concerned with consequences of the regularity assumptions that are imposed by Hypotheses \ref{V-4} and \ref{V-5}. Several results of Subsection \ref{subsec:smooth} are taken from \cite{Ya92_01,Ya10_01}. In Subsection \ref{subsec:diss}, we recall the existence and basic properties of the wave and scattering operators in dissipative scattering theory; Proofs can be found in \cite{Ma75_01,Da78_01,Da80_01,FaFr18_01}.

\subsection{Strong smoothness}\label{subsec:smooth}

Recall that Hypothesis \ref{V-4} assumes that the operator $G:\mathcal{H} \to \mathcal{G}$ is strongly $H_0$-smooth on any compact set $X \Subset \Lambda$, with exponent $s_0 \in (\frac12,1)$, in the sense that, for all $f \in \mathcal{G}$, $\mathcal{F}_0  G_X^* f \in L^2( X ; \mathcal{M} )$ admits a representant belonging to $\mathrm{C}^{s_0}(X; \mathcal{M})$, and that there exists $\mathrm{c}_X>0$ such that, for all $f \in \mathcal{G}$ and $\lambda,  \mu \in X$,
\begin{equation*}
\big \| \mathcal{F}_0  G_X^* f(\lambda) \big \|_{\mathcal{M}} \leq \mathrm{c}_X \|f\|_{\mathcal{G}}, \ \
\big \|Ê\mathcal{F}_0  G_X^* f(\lambda) - \mathcal{F}_0  G_X^* f(\mu) \big \|_{\mathcal{M}} \leq \mathrm{c}_X | \lambda - \mu |^{s_0} \|f\|_{\mathcal{G}} ,
\end{equation*}
where $G_X = G E_0(X)$. In particular, we see that, for all $\lambda \in \Lambda$, the operator $Z_0(\lambda ; G) : \mathcal{G} \to \mathcal{M}$, defined by the relation
\begin{equation}\label{eq:Z0}
Z_0(\lambda;G) f = (\mathcal{F}_0 G^* f) (\lambda),
\end{equation}
is bounded, and the map $\Lambda \in \lambda \mapsto Z_0(\lambda;G) \in \mathcal{L}( \mathcal{G} ; \mathcal{M} )$ is locally Hölder continuous. Note that, according to \eqref{eq:def_Gamma0},
\begin{equation}\label{eq:Z0inGamma}
Z_0(\lambda;G) = \Gamma_0(\lambda) G^* ,
\end{equation}
for a.e. $\lambda \in \Lambda$.

We recall that the wave operators for the self-adjoint pair $(H_V,H_0)$ are defined by
\begin{equation}\label{eq:wave_unitary}
W_\pm( H_V , H_0 ) := \underset{t\to \pm \infty }{\slim} e^{ i t H_V } e^{ - i t H_0 }, \quad W_{ \pm } ( H_0 , H_V ) := \underset{t\to \pm \infty }{\slim} e^{ i t H_0 } e^{ - i t H_V } \Pi_{ \mathrm{ac} }( H_V ).
\end{equation}
Here $\mathcal{H}_{ \mathrm{ac} }( H_V )$ denotes the absolutely continuous spectral subspace of $H_V$ and $\Pi_{ \mathrm{ac} }( H_V )$ is the orthogonal projection onto $\mathcal{H}_{ \mathrm{ac} }( H_V )$. Likewise, we denote by $\mathcal{H}_{ \mathrm{pp} }( H_V )$ the pure point spectral subspace of $H_V$.

Combined with Hypotheses \ref{V-1}--\ref{V-3}, Hypothesis \ref{V-4} has several consequences that will be important in our analysis. We summarize them in the following proposition, and recall a few arguments of their proof for the convenience of the reader.
\begin{proposition}\label{prop:standard_csq}
Suppose that Hypotheses \ref{V-1}--\ref{V-4} hold. Then the following properties are satisfied:
\begin{enumerate}[label=\rm{(\roman*)},leftmargin=*]
\item For all compact set $X \Subset \Lambda$, $G$ is $H_0$-smooth on $X$ in the usual sense of Kato \cite{Ka65_01}, i.e., there exists a constant $\tilde{\mathrm{c}}_X > 0$, such that
\begin{align}\label{eq:H0smooth}
\int_{ \mathbb{R} } \big \| G e^{ - i t H_0 }  u \big \|_{\mathcal{G}}^2 dt \le \tilde{\mathrm{c}}_X^2 \|  u \|_{\mathcal{H}}^2 ,
\end{align}
for all $u \in E_{0}(X) \mathcal{H}$.
\item The maps
\begin{equation*}
\{ z \in \mathbb{C} ,  \, \mathrm{Re}(z) \in  \mathring{\Lambda} , \, \pm \mathrm{Im}(z) \ge 0 \} \ni z \mapsto GR_\sharp(z)G^* \in \mathcal{L}^\infty(\mathcal{G}) ,
\end{equation*}
where $R_\sharp$ stands for $R_0$ or $R_V$, and, in particular, $\mathring{\Lambda} \ni \lambda \mapsto GR_\sharp( \lambda \pm i 0 )G^* \in \mathcal{L}^\infty(\mathcal{G}) $,
are locally H{\"o}lder continuous of order $s_0$.
\item Let $f \in \mathcal{F}_0^* ( \mathrm{C}_0^{\infty}(\mathring{\Lambda}, \mathcal{M}))$. The maps
\begin{equation*}
\{ z \in \mathbb{C} ,  \, \mathrm{Re}(z) \in \mathring{\Lambda} , \, \pm \mathrm{Im}(z) \ge 0 \} \ni z \mapsto GR_\sharp(z) f \in \mathcal{G} ,
\end{equation*}
where $R_\sharp$ stands for $R_0$ or $R_V$, and, in particular, $\mathring{\Lambda} \ni \lambda \mapsto GR_\sharp(\lambda \pm i0) f \in \mathcal{G}$, are continuous.
\item The wave operators \eqref{eq:wave_unitary} exist and are asymptotically complete, i.e.,
\begin{align*}
& \mathrm{Ran} ( W_\pm( H_V , H_0 ) ) = \mathcal{H}_{ \mathrm{ac} }( H_V ) = \mathcal{H}_{ \mathrm{pp} }( H_V )^\perp , \\
& \mathrm{Ran} ( W_\pm( H_0 , H_V ) ) = \mathcal{H} .
\end{align*}
\item For $f \in \mathcal{F}_0^* ( \mathrm{C}_0^{\infty}(\mathring{\Lambda}, \mathcal{M}))$ and a.e. $\lambda \in \Lambda$, let
\begin{equation}\label{eq:gammapm}
\Gamma_{\pm}(\lambda)f := \Gamma_0(\lambda)(I-VR_V(\lambda\pm i0)) f.
\end{equation}
Then the operator $\mathcal{F}_{\pm} : \mathcal{F}_0^* ( \mathrm{C}_0^{\infty}(\mathring{\Lambda}, \mathcal{M})) \to L^2(\Lambda, \mathcal{M})$ defined by
\begin{equation*}
\mathcal{F}_{\pm}f (\lambda) = \Gamma_{\pm}(\lambda) f
\end{equation*}
extends to a unitary map from $\mathcal{H}$ to $L^2(\Lambda, \mathcal{M})$ satisfying $ W_\pm( H_V , H_0 ) = \mathcal{F}_\pm^* \mathcal{F}_0 $.
\item For all $X \Subset \Lambda$, the map
\begin{equation*}
\mathcal{F}_\pm G_X^* : \mathcal{G} \rightarrow \mathrm{C}^{s_0}(X; \mathcal{M}) \text{ is continuous}.
\end{equation*}
In particular, the map
\begin{equation}\label{eq:ZV}
\Lambda \in \lambda \mapsto Z_V^\pm(\lambda;G) := \Gamma_{\pm}(\lambda) G^* \in \mathcal{L}( \mathcal{G} ; \mathcal{M} ),
\end{equation}
is locally Hölder continuous of order $s_0$ and we have that
\begin{equation}\label{eq:decom_Gamma_pm}
Z_V^\pm(\lambda;G) = Z_0(\lambda;G) - Z_0(\lambda;G) KG R_V(\lambda \pm i0) G^*.
\end{equation}
\item For all $\lambda \in \mathring{\Lambda}$, the operators $Z_0(\lambda;G)$ and $Z_V^\pm(\lambda;G)$ are compact.
 \item The unitary scattering matrices $S_V(\lambda)$ defined by \eqref{eq:def_S(HV,H0)} are given by Kuroda's representation formulae
\begin{align}
S_V(\lambda) &= I - 2i\pi Z_0(\lambda;G) (I-KGR_V(\lambda+i0)G^*) K Z_0^* (\lambda;G), \label{eq:Kuroda} \\
             &= I - 2i\pi Z_0(\lambda;G) (I+KGR_0(\lambda+i0)G^*)^{-1} K Z_0^* (\lambda;G) . \label{eq:Kuroda_2}
\end{align}
Moreover, the map $\mathring{\Lambda} \ni \lambda \mapsto S_V(\lambda) \in \mathcal{L}( \mathcal{M} )$ is continuous and for all $\lambda \in \mathring{\Lambda}$, $S_V(\lambda)-I$ is compact.
\end{enumerate}
\end{proposition}
\begin{proof}
(i) is a direct consequence of the definition of strong $H_0$-smoothness together with the fact that, since $X$ is compact, \eqref{eq:H0smooth} is equivalent to
\begin{equation*}
\| G E_X(H) \| < \infty,
\end{equation*}
(see \cite{Ka65_01} or \cite[Theorem XIII.25]{ReSi79_01}).

(ii) The statements for the free resolvent ($R_\sharp = R_0$) are proven in \cite[Theorem 4.4.7]{Ya92_01}. The corresponding statements for $R_V$ are consequences of the resolvent identity
\begin{equation}\label{eq:resolvent}
GR_V(z)G^* = G R_0 (z) G^* (I+ KGR_0(z)G^*)^{-1} \ ,\ {\rm{Im}} \ z \not=0 .
\end{equation}
Indeed, since $z \mapsto G R_0 (z) G^*$ is H{\"o}lder continuous on $\{ z \in \mathbb{C} ,  \, \mathrm{Re}(z) \in \mathring{\Lambda} , \, \pm \mathrm{Im}(z) \ge 0 \}$, we deduce from \eqref{eq:resolvent} that $z \mapsto G R_V(z) G^*$ is also H{\"o}lder continuous on
\begin{equation*}
\{ z \in \mathbb{C} ,  \, \mathrm{Re}(z) \in \mathring{\Lambda} , \, \pm \mathrm{Im}(z) \ge 0 \} \setminus \mathcal{N},
\end{equation*}
where
\begin{equation}\label{eq:def_exceptional_N}
\mathcal{N} := \big \{ \lambda \in \Lambda , \, \exists f \in \mathcal{G}\setminus \{ 0 \} , \, f+KGR_0(\lambda \pm i0)G^* f = 0 \big \}.
\end{equation}
Now, it follows from Hypothesis \ref{V-4} and \cite[Lemma 4.7.2]{Ya92_01} that $\mathcal{N} = \sigma_{\mathrm{pp}}(H_V) \cap \Lambda$. Here we use the fact that the map in Hypothesis \ref{V-4} is H{\"o}lder continuous of order $s_0 > 1/2$.  Since $H_V$ does not have eigenvalues embedded in its essential spectrum according to Hypothesis \ref{V-2}, we deduce that $\mathcal{N} = \emptyset$. This proves (ii).

(iii) Given $f \in \mathcal{F}_0^* ( \mathrm{C}_0^{\infty}(\mathring{\Lambda}, \mathcal{M}))$, it follows from the Privalov theorem, (see \cite[Proposition 0.5.9]{Ya10_01}), that the map $\{ z \in \mathbb{C} ,  \, \mathrm{Re}(z) \in \mathring{\Lambda} , \, \pm \mathrm{Im}(z) \ge 0 \} \ni z \mapsto GR_0(z) f \in \mathcal{G}$ is continuous. Using the resolvent identity
\begin{equation}\label{eq:resolventegeneral}
R_V(z) = R_0(z) -R_0(z) G^* (I+ KGR_0(z)G^* )^{-1} KGR_0(z), \ {\rm{Im}} \ z \not=0,
\end{equation}
and the fact that $\mathcal{N} = \emptyset$ (where $\mathcal{N}$ is defined in \eqref{eq:def_exceptional_N}), we see that the map $z\mapsto GR_V(z)f$ possesses the same property.

(iv) Using Parseval's identity, \eqref{eq:resolventegeneral} and (ii)-(iii), it is not difficult to deduce from (i) that $G$ is $H_V$ smooth in the sense of Kato on $X$. The existence of the strong limits
\begin{align*}\label{eq:wave_unitary}
& W_\pm( H_V , H_0 , X ) := \underset{t\to \pm \infty }{\slim} E_V(X) e^{ i t H_V } e^{ - i t H_0 } E_0(X), \\
& W_{ \pm } ( H_0 , H_V , X ) := \underset{t\to \pm \infty }{\slim} E_0(X) e^{ i t H_0 } e^{ - i t H_V } E_V( X ) ,
\end{align*}
then follow from standard arguments. Proceeding as in the proof of \cite[Theorem 4.5.6]{Ya92_01} and using a density argument, one then obtains the existence and completeness of the wave operators $W_\pm( H_V , H_0 )$ and $W_{ \pm } ( H_0 , H_V )$.

(v) is proven in \cite[Theorem 5.6.1]{Ya92_01}, (see also \cite[Theorem 0.6.12]{Ya10_01}).

(vi) In view of \eqref{eq:gammapm}, the continuity of the map $\mathcal{F}_\pm G_X^* : \mathcal{G} \rightarrow \mathrm{C}^{s_0}(X; \mathcal{M})$ is a direct consequence of Hypothesis \ref{V-4} and (ii). The equation \eqref{eq:decom_Gamma_pm} follows straightforwardly from the decomposition $V=GKG^*$ (see Hypothesis \ref{V-3}) and the definition \eqref{eq:Z0} of $Z_0(\lambda;G)$.

(vii) The identity
\begin{equation*}
Z_0(\lambda;G)^* Z_0(\lambda;G) = G \Gamma_0(\lambda)^* \Gamma_0(\lambda) G^* = \frac{1}{2i\pi} G (R_0(\lambda+i0)-R_0(\lambda-i0)) G^* ,
\end{equation*}
together with (ii), shows that $Z_0(\lambda;G)^* Z_0(\lambda;G)$ is compact. Hence $Z_0(\lambda;G)$ is compact. The same holds for $Z_V^\pm(\lambda;G)$ by (vi).

(viii) The representations \eqref{eq:Kuroda}--\eqref{eq:Kuroda_2} are proven in \cite[Theorem 5.4.4']{Ya92_01} for a.e. $\lambda \in \Lambda$. Continuity of the map  $\mathring{\Lambda} \ni \lambda \mapsto S_V(\lambda) \in \mathcal{L}( \mathcal{M} )$ then follows from (ii) and the continuity of $\lambda \mapsto Z_0(\lambda,G)$. (See also \cite[Theorem 0.7.1]{Ya10_01}). Compactness of $S_V(\lambda)-I$ is a consequence of (vii).
\end{proof}

The main purpose of the following proposition is to show that, assuming Hypothesis \ref{V-5}, properties analogous to (ii)--(iii) in Proposition \ref{prop:standard_csq} hold for the resolvent of $H$ in the region $\{ z \in \mathbb{C} ,  \, \mathrm{Re}(z) \in  \mathring{\Lambda} , \, \mathrm{Im}(z) \ge 0 \}$. Observe that, by Hypothesis \ref{V-5}, the map
\begin{equation*}
\Lambda \in \lambda \mapsto Z_V^\pm(\lambda;C) = \Gamma_\pm(\lambda) C^* \in \mathcal{L}( \mathcal{H} ; \mathcal{M} ),
\end{equation*}
(see \eqref{eq:ZV}), is locally Hölder continuous of order $s$.
\begin{proposition}\label{prop:standard_csq_2}
Suppose that Hypotheses \ref{V-1}--\ref{V-5} hold. Then the following properties are satisfied:
\begin{enumerate}[label=\rm{(\roman*)},leftmargin=*]
\item The maps
\begin{align*}
& \{ z \in \mathbb{C} ,  \, \mathrm{Re}(z) \in  \mathring{\Lambda} , \, \pm \mathrm{Im}(z) \ge 0 \} \ni z \mapsto CR_V(z)C^* \in \mathcal{L}^\infty(\mathcal{H}) ,
\end{align*}
and, in particular, $\mathring{\Lambda} \ni \lambda \mapsto CR_V( \lambda \pm i 0 )C^* \in \mathcal{L}^\infty(\mathcal{H})$, are locally H{\"o}lder continuous of order $s$.
\item  For $\mathrm{Im}(z) \geq 0$, the operator $I -iCR_V(z)C^*$ is invertible in $\mathcal{L}(\mathcal{H})$ and the map
\begin{equation*}
\{ z \in \mathbb{C} , \, \mathrm{Im}(z) \ge 0 \} \in z \mapsto (I -iCR_V(z)C^*)^{-1} \in \mathcal{L}( \mathcal{H} ),
\end{equation*}
is locally H{\"o}lder continuous of order $s$.
\item The maps
\begin{align*}
& \{ z \in \mathbb{C} ,  \, \mathrm{Re}(z) \in  \mathring{\Lambda} , \, \mathrm{Im}(z) \ge 0 \} \ni z \mapsto CR(z)C^* \in \mathcal{L}^\infty(\mathcal{H}) ,
\end{align*}
and, in particular, $\mathring{\Lambda} \ni \lambda \mapsto CR( \lambda + i 0 )C^* \in \mathcal{L}^\infty(\mathcal{H})$, are locally H{\"o}lder continuous of order $s$.
\item $C$ is $H_V$-smooth, i.e., there exists a constant $\mathrm{c}> 0$, such that
\begin{align}\label{eq:HVsmooth}
\int_{ \mathbb{R} } \big \| C e^{ - i t H_V }  u \big \|_{\mathcal{H}}^2 dt \le \mathrm{c}^2 \|  u \|_{\mathcal{H}}^2 ,
\end{align}
for all $u \in \mathcal{H}_{\mathrm{ac}}(H_V)$.
\item For all $\lambda \in \Lambda$, the operators $Z_V^\pm(\lambda;C)$ are compact.
\item Let $f \in \mathcal{F}_\pm^* ( \mathrm{C}_0^{\infty}(\mathring{\Lambda}, \mathcal{M}))$. The maps
\begin{equation*}
\{ z \in \mathbb{C} ,  \, \mathrm{Re}(z) \in \mathring{\Lambda} , \, \mathrm{Im}(z) \ge 0 \} \ni z \mapsto CR_\sharp(z) f \in \mathcal{H} ,
\end{equation*}
where $R_\sharp$ stands for $R_V$ or $R$, and, in particular, $\mathring{\Lambda} \ni \lambda \mapsto CR_\sharp( \lambda + i 0 ) f \in \mathcal{H} $,
are continuous.
\end{enumerate}
\end{proposition}
\begin{proof}
(i) In the same way as in the proof of Proposition \ref{prop:standard_csq}, (i) is proven in \cite[Theorem 4.4.7]{Ya92_01}

(ii) Suppose that $z \in \mathbb{C}$ with $\mathrm{Im}(z)>0$. According to Hypothesis \ref{V-5}, $CR_V(z)C^*$ is compact. Hence, by the Fredholm alternative,
it suffices to prove that $A(z):=I -iCR_V(z)C^*$ is injective. We compute
\begin{align*}\label{eq:injective}
  2 \, {\rm{Re}} \,  A(z) &= A(z) + A(z)^* \\
                      &=2I - iC (R_V (z) - R_V (\bar{z})) C^*  \\
                      &= 2I + 2 \ {\rm{Im}}( z ) \ C R_V (z) R_V (\bar{z}) C^* \geq 2I,
\end{align*}
in the sense of the operators. A continuity argument then implies that $ 2 \, {\rm{Re}} \,  A(z) \geq 2I$ for any $z \in \mathbb{C}$ such that $\mathrm{Im}(z) \geq 0$. This proves that $A(z)$ is invertible. The H{\"o}lder continuity of $z \mapsto A(z)^{-1}$ then follows easily from (i).

(iii) By (ii) and the the resolvent identity, we have that
\begin{equation}\label{eq:resolvent1}
CR(z)C^* = C R_V (z) C^* (I-iCR_V(z)C^*)^{-1} ,\quad {\rm{Im}} (z) >0.
\end{equation}
Since $C R_V (z) C^*$ is compact by Hypothesis \ref{V-5}, this implies that $C R (z) C^*$ is also compact. Moreover, using (i) and again (ii), we see that $z \mapsto CR(z)C^*$ is Hölder continuous of order $s$ on $\{ z \in \mathbb{C} ,  \, \mathrm{Re}(z) \in  \mathring{\Lambda} , \, \mathrm{Im}(z) \ge 0 \}$.

(iv) By \cite[Theorem XIII.25]{ReSi79_01}, it suffices to show that the map
\begin{equation*}
\{ z \in \mathbb{C} , \, \mathrm{Im}(z) \neq 0 \} \ni z \mapsto C \big ( R_V( z ) - R_V( \bar{z} ) \big ) \Pi_{\mathrm{ac}}(H_V) C^* \in \mathcal{L}(\mathcal{H})
\end{equation*}
is bounded. Since, by Hypothesis \ref{V-2}, $\sigma_{\mathrm{ac}}(H_V)=\Lambda$, and since $H_V$ has at most finitely many eigenvalues in $\mathbb{R} \setminus \Lambda$, it suffices in turn to verify that
\begin{equation*}
\{ z \in \mathbb{C} , \, \mathrm{Re}(z) \in \Lambda, \, \mathrm{Im}(z) \neq 0 \} \ni z \mapsto C \big ( R_V( z ) - R_V( \bar{z} ) \big ) C^* \in \mathcal{L}(\mathcal{H})
\end{equation*}
is bounded. This follows from the facts that the limits $C R_V( \lambda \pm i 0 ) C^* = \lim_{\varepsilon \downarrow 0} C R_V( \lambda \pm i \varepsilon ) C^*$ exist for all $\lambda \in \mathring{\Lambda}$, by (i), and that
\begin{equation*}
\mathring{\Lambda} \in \lambda \mapsto C \big ( R_V( \lambda + i 0 ) - R_V( \lambda - i 0 ) \big ) C^* \in \mathcal{L}( \mathcal{H} ) ,
\end{equation*}
is bounded by Hypothesis \ref{V-5}.

(v) It suffices to use the same argument as in the proof of Proposition \ref{prop:standard_csq} (vii), using (iv).

(vi) Let $f \in \mathcal{F}_{\pm}^* \ ( \mathrm{C}_0^{\infty}(\mathring{\Lambda}, \mathcal{M}))$. As in the proof or Proposition \ref{prop:standard_csq} (iii), the map $z \mapsto CR_V(z) f$ is continuous on $\{ z \in \mathbb{C} ,  \, \mathrm{Re}(z) \in \mathring{\Lambda} , \, \mathrm{Im}(z) \ge 0 \}$. Using (ii), (iii) and the resolvent identity
\begin{equation}\label{eq:resolgen}
R(z) = R_V(z) +iR_V(z) C^* (I-iCR_V(z)C^* )^{-1} CR_V(z), \quad  {\rm{Im}}(z) \geq0,
\end{equation}
we see that $z \mapsto CR(z)f$ is also continuous on  $\{ z \in \mathbb{C} ,  \, \mathrm{Re}(z) \in \mathring{\Lambda} , \, \mathrm{Im}(z) \ge 0 \}$.
\end{proof}

To conclude this section, we mention that the following resolvent formula holds under Hypothesis \ref{V-1}--\ref{V-5}:
\begin{equation}\label{eq:PALH}
I+iCR(\lambda+i0)C^* =  (I-iCR_V(\lambda+i0)C^*)^{-1} ,\quad \lambda \in \mathring{\Lambda}.
\end{equation}
Equation \eqref{eq:PALH} will be used several times in Section \ref{sec:proof}.

\subsection{Dissipative scattering theory.}\label{subsec:diss}

In this section, we recall some of the results obtained in \cite{FaFr18_01} (see also \cite{Da78_01,Da80_01,Ma75_01}) that will be used throughout Section \ref{sec:proof}. Note that these results are actually proven under slightly weaker assumptions in \cite{FaFr18_01}.
For instance, in \cite{FaFr18_01}, it is not supposed that $G$ is strongly $H_0$-smooth, only that the wave operators $W_{\pm}(H_V, H_0)$ exist and are asymptotically complete. Likewise, in \cite{FaFr18_01},
the operator $C$ is only assumed to be $H_V$-smooth in the usual sense of Kato, while we assume here that $C$ is strongly $H_V$-smooth.

As mentioned before, the wave operators $W_-( H , H_0 )$ and $W_+( H^* , H_0 )$ are defined by
\begin{equation*}
W_-( H , H_0 ) = \underset{t\to \infty }{\slim}  \, e^{ - i t H } e^{ i t H_0 }, \quad W_+( H^* , H_0 ) = \underset{t\to \infty }{\slim}  \, e^{ i t H^* } e^{ - i t H_0 },
\end{equation*}
provided that the strong limits exist. Following \cite{Da78_01,Da80_01}, one can define the {\it{absolutely continuous subspace}} for the dissipative operator $H$ by setting
\begin{equation*}
\mathcal{H}_{ \mathrm{ac} }( H ) := \overline{ M ( H ) },
\end{equation*}
where
\begin{equation*}
 M( H ) := \Big \{ u \in \mathcal{H} , \exists \mathrm{c}_u > 0 ,  \forall v \in \mathcal{H} , \int_0^\infty \big |\langle e^{ - i t H }u , v \rangle \big |^2 dt \le \mathrm{c}_u \| v \|^2 \Big \} .
\end{equation*}
Note that when $H$ is self-adjoint, this definition coincides with the usual one  based on the nature of the spectral measures of $H$. The absolutely continuous subspace for $H^*$ is defined in the same way, replacing $e^{ - i t H }$ by $e^{ i t H^* }$ in the equation above. It is proven in \cite{Da80_01} that, under assumptions weaker than Hypotheses \ref{V-1}--\ref{V-5},
\begin{equation*}
\mathcal{H}_{ \mathrm{ac} }( H ) = \mathcal{H}_{ \mathrm{b} }( H )^\perp ,
\end{equation*}
and likewise for $\mathcal{H}_{ \mathrm{ac} }( H^* )$, where $\mathcal{H}_{ \mathrm{b} }( H )$, the space of bound states for $H$, is defined in \eqref{eq:defHb}. Since $\mathcal{H}_{ \mathrm{b} }( H )=\mathcal{H}_{ \mathrm{b} }( H^* )$ (see \cite[Lemma 1]{Da80_01} or \cite[Lemma 3.1]{FaFr18_01}), it follows that $\mathcal{H}_{ \mathrm{ac} }( H ) = \mathcal{H}_{ \mathrm{ac} }( H^* )$. We also recall that the dissipative spaces $\mathcal{H}_{ \mathrm{d} }( H )$ and $\mathcal{H}_{ \mathrm{d} }( H^* )$ are defined in \eqref{eq:defHd}--\eqref{eq:defH*d}.

The results of Section \ref{subsec:smooth} show that Hypotheses \ref{V-1}--\ref{V-5} of the present paper imply Hypotheses 2.1, 2.3 and 2.4 made in \cite{FaFr18_01}. Therefore, \cite[Proposition 3.4]{FaFr18_01} (see also \cite{Da80_01}) implies the following result.
\begin{proposition}\label{prop:existence_W-}
Suppose that Hypotheses \ref{V-1}--\ref{V-5} hold. The wave operators $W_-( H , H_0 )$ and $W_+(H^*,H_0)$ exist, are injective contractions, and  satisfy
the intertwining properties
  \begin{align}
 & H W_-( H , H_0)  = W_-( H , H_0 ) H_0 \quad {\text{on}} \quad     \mathcal{D}(H_0), \label{eq:inter-2} \\
 & H^* W_+( H^* , H_0)  = W_+( H^* , H_0 ) H_0 \quad  {\text{on}} \quad  \mathcal{D}(H_0).  \label{eq:inter-21}
  \end{align}
Moreover,
  \begin{align*}
& \overline{\mathrm{Ran}( W_-( H , H_0 ) ) } = \big ( \mathcal{H}_{ \mathrm{b} }( H ) \oplus \mathcal{H}_{ \mathrm{d} }( H^* ) \big )^\perp \subset M(H) \subset \mathcal{H}_{ \mathrm{ac} }( H ) , \\
& \overline{\mathrm{Ran}( W_+( H^* , H_0 ) ) } = \big ( \mathcal{H}_{ \mathrm{b} }( H ) \oplus \mathcal{H}_{ \mathrm{d} }( H ) \big )^\perp \subset M(H^*) \subset \mathcal{H}_{ \mathrm{ac} }( H ).
\end{align*}
\end{proposition}

The  wave operators $W_+( H_0 , H )$ and $W_-( H_0 , H^* )$ are defined by
\begin{align*}
& W_+( H_0 , H ) := \underset{t\to \infty }{\slim}  \, e^{ i t H_0 } e^{ - i t H } \Pi_{ \mathrm{ac} }( H ), \\
& W_-( H_0 , H^* ) := \underset{t\to \infty }{\slim}  \, e^{ - i t H_0 } e^{ i t H^* } \Pi_{ \mathrm{ac} }( H ),
\end{align*}
provided that the strong limits exist, where $\Pi_{ \mathrm{ac} }( H )$ is the orthogonal projections onto $\mathcal{H}_{ \mathrm{ac} }( H )$. The following result is proven in \cite[Proposition 3.6]{FaFr18_01}.
\begin{proposition}\label{prop:existence_W+}
Suppose that Hypotheses \ref{V-1}--\ref{V-5} hold. The wave operators $W_+( H_0 , H )$ and $W_-( H_0 , H^* )$ exist, are contractions with dense ranges, and their kernels are given by
\begin{align*}
& \mathrm{Ker}( W_+( H_0 , H ) ) = \mathcal{H}_{ \mathrm{b} }( H ) \oplus \mathcal{H}_{ \mathrm{d} }( H ) , \\
& \mathrm{Ker}( W_-( H_0 , H^* ) ) = \mathcal{H}_{ \mathrm{b} }( H ) \oplus \mathcal{H}_{ \mathrm{d} }( H^* ).
\end{align*}
Moreover,
\begin{align*}
W_+^*( H_0 , H ) = W_+( H^* , H_0 ), \quad W_-^*( H_0 , H^* ) = W_-( H , H_0 ) ,
\end{align*}
and we have the intertwining properties
\begin{align}
& H_0 W_+( H_0 , H)  = W_+( H_0 , H ) H  \quad {\text{on}} \quad  \mathcal{D}(H) , \label{eq:inter-4} \\
& H_0 W_-( H_0 , H^*)  = W_-( H_0 , H^* ) H^*  \quad {\text{on}} \quad  \mathcal{D}(H). \label{eq:inter-41}
\end{align}
\end{proposition}
By Propositions \ref{prop:existence_W-} and \ref{prop:existence_W+}, the scattering operators $S(H,H_0)$ in \eqref{eq:def_scatt_op} is well-defined and satisfies
\begin{equation}
S( H , H_0 ) = W_+( H_0 , H ) W_-( H , H_0 ). \label{eq:defscattop}
\end{equation}
Moreover, $S(H,H_0)$ is a contraction and commutes with $H_0$. Its adjoint is given by
\begin{equation*}
S( H , H_0 )^* = S( H^* , H_0 ).
\end{equation*}

We conclude this section by recalling one of the main results of \cite{FaFr18_01} which should be compared to Theorem \ref{thm:invert_S} of the present paper. As in \cite{Da80_01,FaFr18_01}, we set
\begin{align}
& \mathcal{H}_{ \mathrm{p} }( H ) := \mathrm{Span} \, \big \{ u \in \mathrm{Ran}( \Pi_\lambda ) , \, \lambda \in \sigma( H ) , \, \mathrm{Im} \, \lambda < 0  \big \} , \label{eq:defHp} \\
& \mathcal{H}_{ \mathrm{p} }( H^* ) := \mathrm{Span} \, \big \{ u \in \mathrm{Ran}( \Pi^*_\lambda ) , \, \lambda \in \sigma( H^* ) , \, \mathrm{Im} \, \lambda > 0  \big \} ,\label{eq:defH^*p}
\end{align}
where $\Pi_\lambda$, respectively $\Pi^*_\lambda$, denotes the Riesz projection corresponding to an isolated eigenvalue of $H$, respectively of $H^*$. The following result is established in \cite{FaFr18_01}.
\begin{theorem}\label{thm:mainFF}
Assume that Hypotheses \ref{V-1}--\ref{V-5} hold. Suppose that $H$ has finitely many eigenvalues, no spectral singularities in $\Lambda$, and that there exists $m > 0$ such that
\begin{equation*}
\sup_{\mu \ge m , \, \varepsilon > 0 } \big \| C \big ( H - ( \mu - i \varepsilon ) \big )^{-1} C^* \big \| < \infty.
\end{equation*}
Then $\mathcal{H}_{ \mathrm{p} }( H ) = \mathcal{H}_{ \mathrm{d} }( H )$, $\mathcal{H}_{ \mathrm{p} }( H^* ) = \mathcal{H}_{ \mathrm{d} }( H^* )$, and
\begin{align}\label{eq:RanW-FF}
&Ê\mathrm{Ran}( W_-( H , H_0 ) ) = \big ( \mathcal{H}_{ \mathrm{b} }( H ) \oplus \mathcal{H}_{ \mathrm{p} }( H^* ) \big )^\perp.
\end{align}
Moreover, $W_+( H_0 , H )$ is surjective and $S( H, H_0 )$ is bijective.
\end{theorem}
Clearly, the expression of $\mathrm{Ran}( W_-( H , H_0 ) )$ obtained in \eqref{eq:RanW-FF} is more precise than that given by \eqref{eq:RanW-closed} of Theorem \ref{thm:invert_S}.  However, Theorem \ref{thm:mainFF} requires that $H$ have only finitely many eigenvalues, which we do not impose in Theorem \ref{thm:invert_S}. For dissipative Schr{\"o}dinger operators of the form \eqref{eq:a1}, unless $V$ and $W$ are exponentially decaying, $H$ may have infinitely many eigenvalues.

\section{Proofs of the main results.}\label{sec:proof}

In this section we prove our main results: Theorem \ref{thm:RF} is proven in Section \ref{subsec:thm1}, Theorem \ref{thm:invert} in Section \ref{subsec:thminvert} and Theorem \ref{thm:invert_S} in Section \ref{subsec:thminvertS}.

\subsection{Proof of Theorem \ref{thm:RF}}\label{subsec:thm1}

In this section, we prove Kuroda's representation formula for the scattering matrices $S(\lambda)$. Our proof follows the usual approach used in the self-adjoint case, see e.g. \cite{IsKi85_01},
based on the notion of abelian integrals. To simplify the presentation, in the proof below, we only consider the case  where $V=0$. Theorem \ref{thm:RF} in the general case where $V$ does not necessarily vanishes can be obtained in the same manner, using in addition the ``chain rule'', (see e.g. \cite[Section 10]{Ya00_01}).

Before going into the proof, we recall a few elementary facts. Since $H$ is a dissipative operator, we have that $\| e^{-itH} \| \leq 1$ for all $t \geq 0$ and, for all $u \in \mathcal{H}$,
\begin{equation}\label{eq:integrale1}
\int_0^{+\infty} \| C e^{-itH} u \|^2 dt \leq \frac{1}{2} \|u\|^2,
\end{equation}
(see \cite{FaFr18_01}; here we used the fact that we consider the simplified setting in which $V=0$). Moreover, by Proposition \ref{prop:standard_csq_2} (iv), assuming Hypotheses \ref{V-1} and \ref{V-5} with $V=0$, we have that $C$ is $H_0$-smooth, namely there exists $\mathrm{c}>0$ such that
\begin{equation}\label{eq:integral2}
\int_{-\infty}^{+\infty} \| C e^{-itH_0} u \|^2  dt \leq \mathrm{c} \|u\|^2,
\end{equation}
for all $u \in \mathcal{H}$. Equations \eqref{eq:integrale1}--\eqref{eq:integral2} allow us to express the wave operators $W_- (H, H_0)$ and $W_+(H_0, H)$ as (see, e.g., \cite[Appendix B]{FaFr18_01}),
\begin{align}
& W_- (H,H_0) = I - \int_0^{+\infty} e^{-itH} C^* C e^{itH_0}  dt , \label{eq:waveint} \\
& W_+(H_0,H) = I - \int_0^{+\infty} e^{itH_0} C^* C e^{itH}  dt. \label{eq:waveint_2}
\end{align}
Furthermore, the resolvents of $H_0$ and $H$ can be expressed, for all $\varepsilon >0$, by the well-known formulae
\begin{equation}\label{eq:resolvint}
R_0(\lambda+i\varepsilon) = i \int_0^{+\infty} e^{-\varepsilon  t} e^{-it(H_0 -\lambda)} dt , \quad R(\lambda+i\varepsilon) = i \int_0^{+\infty} e^{-\varepsilon  t} e^{-it(H -\lambda)} dt.
\end{equation}

Now we are ready to prove Theorem \ref{thm:RF}.

\begin{proof}[Proof of Theorem \ref{thm:RF}]
Suppose that $V=0$ and recall from Remark \ref{rk:kuroda} (i) that the scattering matrices for the pair $(H_0-iC^*C,H_0)$ are denoted by $\tilde{S}(\lambda)$. Let $f,g \in \mathcal{F}_{0}^* ( \mathrm{C}_0^{\infty}(\mathring{\Lambda}, \mathcal{M}))$. Using the direct integral representation \eqref{eq:ST1}, we have that
\begin{equation} \label{eq:intmatrices}
\big \langle (S(H,H_0) -1)f,g \big \rangle_{\mathcal{H}}  = \int_{\Lambda} \big \langle ( \tilde{S}(\lambda)-1) \Gamma_0(\lambda)f, \Gamma_0(\lambda)f \big \rangle_{\mathcal{M}} \ d\lambda.
\end{equation}
We decompose
\begin{align}
\big \langle(S(H,H_0) -1)f,g \big \rangle_{\mathcal{H}} &= \big \langle W_+(H,H_0) (W_- (H,H_0)-1)f,g \big \rangle_{\mathcal{H}} + \big \langle (W_+(H_0,H)-1)f,g \big \rangle_{\mathcal{H}} \notag \\
&=: \mathrm{(I)} + \mathrm{(II)}. \label{eq:defI-II}
\end{align}

\vspace{0.2cm}\noindent
\textbf{Step 1.} We begin by studying the second term in \eqref{eq:defI-II}. Using \eqref{eq:waveint_2} and \eqref{eq:resolvint}, we obtain that
\begin{align*}
\mathrm{(II)} &=- \int_0^{+\infty} \big \langle e^{itH_0} C^* C e^{-itH}f,g \big \rangle_{\mathcal{H}}  dt \\
                     &= - \int_0^{+\infty} \int_\Lambda \big \langle \Gamma_0(\lambda)  C^* C e^{-it(H-\lambda)}f, \Gamma_0(\lambda) g \big \rangle_{\mathcal{M}} d\lambda  dt \\
                     &= i\   \lim_{ \varepsilon \downarrow 0 }\int_\Lambda \bigg \langle\Gamma_0(\lambda)  C^* C
                            \left( i \int_0^{+\infty} e^{-t\varepsilon}  e^{-it(H-\lambda)}f  dt \right), \Gamma_0(\lambda) g \bigg \rangle_{\mathcal{M}} d\lambda \\
                     &=  i\ \lim_{ \varepsilon \downarrow 0 } \int_\Lambda \big \langle \Gamma_0(\lambda)
                           C^* C R(\lambda+i \varepsilon)f, \Gamma_0(\lambda) g \big \rangle_{\mathcal{M}}  d\lambda \\
                     &=  i\ \int_\Lambda \big \langle \Gamma_0(\lambda)  C^* C R(\lambda+i0)f, \Gamma_0(\lambda) g \big \rangle_{\mathcal{M}} d\lambda.
\end{align*}
In the last equality, we used that $CR(\lambda + i0)f$ is well-defined for $f \in \mathcal{F}_{0}^* ( \mathrm{C}_0^{\infty}(\mathring{\Lambda}, \mathcal{M}))$, see Proposition \ref{prop:standard_csq_2} (vi).

\vspace{0.2cm}\noindent
\textbf{Step 2.} Next we study the first term in \eqref{eq:defI-II}. Using \eqref{eq:waveint} and the intertwining property \eqref{eq:inter-4}, we compute
\begin{align*}
\mathrm{(I)} &=  - \bigg \langle W_+ (H_0, H) \left(\int_0^{+\infty} e^{-itH} C^* C e^{itH_0} f \ dt \right), g \bigg \rangle_{\mathcal{H}} \\
    &=  - \int_0^{+\infty}  \big \langle e^{-itH_0} W_+ (H_0, H) C^* C e^{itH_0} f  , g \big \rangle_{\mathcal{H}}  dt\\
    &= - \int_0^{+\infty} \big \langle e^{-itH_0}  C^* C e^{itH_0} f  , g \big \rangle_{\mathcal{H}}  dt  - \int_0^{+\infty} \big \langle e^{-itH_0} (W_+ (H_0, H)-1) C^* C e^{itH_0} f  , g \big \rangle_{\mathcal{H}}  dt\\
    &=: \mathrm{(a)}+\mathrm{(b)}.
\end{align*}
The direct integral decomposition \eqref{eq:ST1} gives
\begin{align*}
\mathrm{(a)} &=  - \int_0^{+\infty}  \int_\Lambda \big \langle  \Gamma_0(\lambda) C^* C e^{it(H_0-\lambda)} f  , \Gamma_0(\lambda) g \big \rangle_{\mathcal{M}}  d\lambda dt \\
    &= -i \lim_{ \varepsilon \downarrow 0 }  \int_\Lambda \bigg \langle   \Gamma_0(\lambda) C^* C
        \left( -i \int_0^{+\infty} e^{-t\varepsilon} e^{it(H_0-\lambda)} f dt \right)
         , \Gamma_0(\lambda) g \bigg \rangle_{\mathcal{M}}  d\lambda \\
    &= -i \lim_{ \varepsilon \downarrow 0 }  \int_\Lambda \big \langle  \Gamma_0(\lambda) C^* C
         R_0(\lambda-i\varepsilon) f , \Gamma_0(\lambda) g \big \rangle_{\mathcal{M}}  d\lambda \\
        &= -i \int_\Lambda \big \langle \Gamma_0(\lambda) C^* C R_0(\lambda-i0) f , \Gamma_0(\lambda) g \big \rangle_{\mathcal{M}}  \ d\lambda,
\end{align*}
where we used \eqref{eq:resolvint} together with the fact that $C R_0(\lambda-i0) f$ is well-defined for $f \in \mathcal{F}_{0}^* ( C_0^{\infty}(\mathring{\Lambda}, \mathcal{M}))$ by Proposition \ref{prop:standard_csq_2} (vi).
A similar calculation shows that
\begin{equation*}
\mathrm{(b)} = \int_\Lambda \big \langle \Gamma_0(\lambda) C^*C R(\lambda+i0) C^* C R_0(\lambda-i0) f, \Gamma_0(\lambda)g \big \rangle_{\mathcal{M}} d\lambda.
\end{equation*}

\vspace{0.2cm}\noindent
\textbf{Step 3.} Putting together the equalities obtained in Steps 1 and 2 yields
\begin{align}
\big \langle(S(H,H_0) -1)f,g\big \rangle_{\mathcal{H}} &=  - i \int_\Lambda \big \langle\Gamma_0(\lambda) C^* C R_0(\lambda-i0) f , \Gamma_0(\lambda)g \big \rangle_{\mathcal{M}}  d\lambda \notag \\
                    &  +   \int_\Lambda \big \langle \Gamma_0(\lambda) C^*C R(\lambda+i0) C^* C R_0(\lambda-i0) f, \Gamma_0(\lambda)g \big \rangle_{\mathcal{M}}  d\lambda \notag \\
                    &  + i \int_\Lambda \big \langle \Gamma_0(\lambda)  C^* C R(\lambda+i0)f, \Gamma_0(\lambda) g \big \rangle_{\mathcal{M}}  d\lambda \notag \\
                    &=: \mathrm{(\alpha)} + \mathrm{(\beta)} + \mathrm{(\gamma)}. \label{eq:alph_beta_gamm}
\end{align}
Recall from \eqref{eq:Z0inGamma} that $Z_0(\lambda;C) = \Gamma_0(\lambda)C^*$ and from Hypothesis \ref{V-5} that $\Lambda \in \lambda \mapsto Z_0(\lambda;C) \in \mathcal{L}( \mathcal{H} ; \mathcal{M} )$ is locally Hölder continuous. We can rewrite $(\alpha)$ as
\begin{align}
 \mathrm{(\alpha)} &= - i \int_\Lambda \big \langle C R_0(\lambda-i0) f , Z_0(\lambda;C)^* \Gamma_0(\lambda)g \big \rangle_{\mathcal{H}}  d\lambda \notag \\
 &= - i \int_\Lambda \big \langle C R_0(\lambda+i0) f , Z_0(\lambda;C)^* \Gamma_0(\lambda)g \big \rangle_{\mathcal{H}}  d\lambda \notag \\
 &\quad + i \int_\Lambda \big \langle C ( R_0(\lambda+i0) - R_0(\lambda-i0) ) f , Z_0(\lambda;C)^* \Gamma_0(\lambda)g \big \rangle_{\mathcal{H}}  d\lambda . \label{eq:alpha}
\end{align}
Using again the direct integral representation \eqref{eq:ST1} gives
\begin{align*}
& \big \langle C ( R_0(\lambda+i0) - R_0(\lambda-i0) ) f , Z_0(\lambda;C)^* \Gamma_0(\lambda)g \big \rangle_{\mathcal{H}} \\
& =\int_\Lambda \big \langle \Gamma_0(\mu) ( R_0(\lambda+i0) - R_0(\lambda-i0) ) f , \Gamma_0(\mu) C^* Z_0(\lambda;C)^* \Gamma_0(\lambda)g \big \rangle_{\mathcal{M}} d\mu \\
& =\int_\Lambda \big \langle ( ( \mu - (\lambda+i0) )^{-1} - ( \mu - (\lambda-i0) )^{-1} )  \Gamma_0(\mu) f , Z_0(\mu;C)  Z_0(\lambda;C)^* \Gamma_0(\lambda)g \big \rangle_{\mathcal{M}} d\mu
\end{align*}
Using that $\mu \mapsto \Gamma_0(\mu) f$ is continuous (because $f \in \mathcal{F}_{0}^* ( \mathrm{C}_0^{\infty}(\mathring{\Lambda}, \mathcal{M}))$) and that $\mu \mapsto Z_0(\mu;C)$ is also continuous by Hypothesis \ref{V-5}, we obtain that
\begin{align*}
& \big \langle C ( R_0(\lambda+i0) - R_0(\lambda-i0) ) f , Z_0(\lambda;C)^* \Gamma_0(\lambda)g \big \rangle_{\mathcal{H}} \\
& = 2i\pi \big \langle \Gamma_0(\lambda) f , Z_0(\lambda;C)  Z_0(\lambda;C)^* \Gamma_0(\lambda) g \big \rangle_{\mathcal{M}} .
\end{align*}
Inserting this into \eqref{eq:alpha} leads to
\begin{align}
 \mathrm{(\alpha)}  &= - i \int_\Lambda \big \langle C R_0(\lambda+i0) f , Z_0(\lambda;C)^* \Gamma_0(\lambda)g \big \rangle_{\mathcal{H}}  d\lambda \notag \\
 &\quad - 2 \pi \int_\Lambda \big \langle \Gamma_0(\lambda) f , Z_0(\lambda;C)  Z_0(\lambda;C)^* \Gamma_0(\lambda) g \big \rangle_{\mathcal{M}}  d\lambda . \label{eq:alpha2}
\end{align}

We proceed in the same way for $(\beta)$, yielding
\begin{align}
\mathrm{(\beta)} &= -2i\pi \int_\Lambda \big \langle Z_0(\lambda;C) C R(\lambda+i0) C^* Z_0(\lambda;C)^* \Gamma_0(\lambda)    f, \Gamma_0(\lambda)g \big \rangle_{\mathcal{M}} d\lambda \notag \\
    & \quad + \int_\Lambda \big \langle Z_0(\lambda;C) C R(\lambda+i0) C^* C R_0(\lambda+i0) f, \Gamma_0(\lambda)g \big \rangle_{\mathcal{M}} d\lambda , \label{eq:beta}
\end{align}
and the resolvent identity gives
\begin{equation}
C R(\lambda+i0) C^* C R_0(\lambda+i0) f = -iC (R(\lambda+i0) - R_0(\lambda+i0)) f. \label{eq:resolvent_final}
\end{equation}

Combining \eqref{eq:alph_beta_gamm}, \eqref{eq:alpha2}, \eqref{eq:beta} and \eqref{eq:resolvent_final}, we obtain that
\begin{align*}
\big \langle(S(H,H_0) -I)f,g \big \rangle_{\mathcal{H}} &= - 2\pi  \int_\Lambda \big \langle Z_0(\lambda ; C ) Z_0(\lambda ; C)^* \Gamma_0(\lambda) f, \Gamma_0(\lambda)g \big \rangle_{\mathcal{M}} d\lambda \\
                    &  -2i\pi \int_\Lambda \big \langle Z_0(\lambda;C)C R(\lambda+i0) C^* Z_0(\lambda;C)^* \Gamma_0(\lambda)    f,
                    \Gamma_0(\lambda)g \big \rangle_{\mathcal{M}} d\lambda .
\end{align*}
Comparing with \eqref{eq:intmatrices}, we obtain Kuroda's representation formula for the scattering matrices for the pair $(H_0-iC^*C,H_0)$:
\begin{equation}
\tilde{S}(\lambda) = I -2\pi Z_0(\lambda;C) (I+iCR(\lambda+i0) C^*) Z_0(\lambda;C)^*.
\end{equation}
This proves \eqref{eq:S0(lambda)1}. To obtain \eqref{eq:S0(lambda)2}, it suffices to apply \eqref{eq:PALH} with $V=0$. Compactness of $\tilde{S}(\lambda)-I$ follows from Proposition \ref{prop:standard_csq_2} (v). Moreover, by Proposition \ref{prop:standard_csq_2} (ii), we see that $\mathring{\Lambda} \ni \lambda \mapsto \tilde{S}(\lambda)$ is continuous.

\vspace{0.2cm}\noindent
\textbf{Step 4.} It remains to verify that $\tilde{S}(\lambda)$ is a contraction. The adjoint of $\tilde{S}(\lambda)$ is given by
\begin{equation}
\tilde{S}^*(\lambda) = I -2\pi Z_0(\lambda;C) (I+iCR_0(\lambda+i0) C^*)^{-1} Z_0(\lambda;C)^*.
\end{equation}
Hence one can compute
\begin{align*}
&\tilde{S}^*(\lambda) \tilde{S}(\lambda) = I - 2\pi Z_0(\lambda;C) \big \{ (I+iCR_0(\lambda-i0) C^*)^{-1} + (I-iCR_0(\lambda+i0) C^*)^{-1} \big \} Z_0(\lambda;C)^* \\
& + 4\pi^2 Z_0(\lambda;C) (I+iCR_0(\lambda-i0) C^*)^{-1} Z_0(\lambda;C)^*Z_0(\lambda;C) (I-iCR_0(\lambda+i0) C^*)^{-1} Z_0(\lambda;C)^*.
\end{align*}
Observing that
\begin{align*}
Z_0(\lambda;C)^*Z_0(\lambda;C) &= \frac{1}{2i\pi} C[ R_0(\lambda+i0) - R_0(\lambda-i0)]C^* \\
&= \frac{1}{2\pi} (I + i C R_0(\lambda-i0) C^* ) + \frac{1}{2\pi} (I - i C R_0(\lambda+i0) C^* ) - \frac{1}{\pi} ,
\end{align*}
we obtain
\begin{align*}
&\tilde{S}^*(\lambda) \tilde{S}(\lambda) = I - 4\pi Z_0(\lambda;C) (I+iCR_0(\lambda-i0) C^*)^{-1} (I-iCR_0(\lambda+i0) C^*)^{-1} Z_0(\lambda;C)^*.
\end{align*}
Therefore, $\tilde{S}^*(\lambda) \tilde{S}(\lambda) - I \le 0$, which implies that $\tilde{S}(\lambda)$ is a contraction. Since $\tilde{S}(\lambda) - I$ is compact, this implies that
\begin{equation*}
\| \tilde{S}(\lambda) \|=1 ,
\end{equation*}
if $\mathrm{dim} \, \mathcal{M} = + \infty$. This concludes the proof of Theorem \ref{thm:RF} in the case where $V=0$. As mentioned above, the general case where $V$ does not necessarily vanishes can be obtained similarly, using in addition the chain rule for scattering matrices \cite[Section 10]{Ya00_01}.
\end{proof}

\subsection{Proof of Theorem \ref{thm:invert}}\label{subsec:thminvert}

As in Subsection \ref{subsec:thm1}, we prove Theorem \ref{thm:invert} in the case where $V=0$. The proof in the case where $V\neq 0$ can be obtained by straightforward modifications.

We begin with a first characterization of the notion of spectral regularity introduced in Definition \ref{def:spec-sing}.
\begin{lemma}\label{lem:criterion}
Suppose that Hypotheses \ref{V-1}--\ref{V-5} hold. Let $\lambda \in \mathring{\Lambda}$. The following conditions are equivalent
\begin{enumerate}[label=\rm{(\roman*)},leftmargin=*]
\item $\lambda$ is a regular spectral point of $H$ in the sense of Definition \ref{def:spec-sing}.
\item $\lambda$ is not an accumulation point of eigenvalues of $H$ located in $\lambda-i(0,\infty)$ and the limit
\begin{equation*}
C R ( \lambda - i 0 ) C^* := \lim_{ \varepsilon \downarrow 0 } C R ( \lambda - i \varepsilon ) C^*
\end{equation*}
exists in the norm topology of $\mathcal{L}( \mathcal{H} )$.
\item $I-iCR_V(\lambda-i0)C^*$ is invertible in $\mathcal{L}( \mathcal{H} )$.
\end{enumerate}
\end{lemma}
\begin{proof}
As above, to simplify the presentation, we prove the lemma for $V=0$. The proof in the case where $V\neq0$ is identical.

(i)$\,\Rightarrow\,$(ii) is obvious. We prove (ii)$\,\Rightarrow\,$(iii). Let $\lambda \in \mathring{\Lambda}$ and suppose that
\begin{equation*}
C R ( \lambda - i 0 ) C^* = \lim_{ \varepsilon \downarrow 0 } C R ( \lambda - i \varepsilon ) C^*
\end{equation*}
exists in $\mathcal{L}( \mathcal{H} )$. Recall  that the limit $CR_0 (\lambda-i0) C^* = \lim_{\varepsilon \downarrow 0 } C R_0( \lambda - i \varepsilon )$ exists and that the operator $CR_0 (\lambda-i0) C^*$ is compact (see Proposition \ref{prop:standard_csq_2} (i)). Using the resolvent identity
\begin{equation}\label{eq:resolvent2}
(I+iCR(z)C^*) (I-iCR_0 (z) C^*) = I \quad {\rm{for }} \ \mathrm{Im}(z) <0 \ {\rm{and}} \ z \notin \sigma(H),
\end{equation}
we deduce immediately from \eqref{eq:resolvent2} that
\begin{equation}
(I+iCR(\lambda-i0)C^*) (I-iCR_0 (\lambda-i0) C^*) = I.
\end{equation}
It follows that $I-iCR_0 (\lambda-i0) C^*$ is injective, and thus invertible thanks to  the Fredholm alternative.

It remains to prove that (iii)$\,\Rightarrow\,$(i). Assume that the operator $A(\lambda):= I-iCR_0(\lambda-i0)C^*$ is invertible in $\mathcal{L}(\mathcal{H})$. Since the map $\lambda \mapsto A(\lambda)$ is continuous by Proposition \ref{prop:standard_csq_2} (i), it follows that $A(\mu)$ is also invertible in $\mathcal{L}(\mathcal{H})$ for all $\mu$ in a compact neighborhood $K_\lambda$ of $\lambda$. We claim that if $\mu\in K_\lambda$, $\mu$ is not an accumulation point of eigenvalues of $H$ located in $\mu - i ( 0 , \infty )$. Indeed, suppose by contradiction that $\mu - i \varepsilon$ is an eigenvalue of $H$, for $\varepsilon > 0$ small enough, and let $u\neq 0$ be such that
\begin{equation}\label{eq:vp}
(H - (\mu - i \varepsilon) )u=0.
\end{equation}
This implies that
\begin{align*}
C R_0(\mu - i \varepsilon ) (H - (\mu - i \varepsilon) )u = ( I - i C R_0( \mu - i \varepsilon ) C^* ) C u = 0.
\end{align*}
Since $I-iCR_0(\mu-i0)C^*$ is invertible in $\mathcal{L}(\mathcal{H})$, $I-iCR_0(\mu-i\varepsilon)C^*$ is also invertible for $\varepsilon>0$ small enough, by continuity. The previous equation therefore yields that $Cu=0$. Inserting this into \eqref{eq:vp} gives $(H_0 - (\mu - i \varepsilon) )u=0$, which is a contradiction since $H_0$ is self-adjoint.

Next, we must show that $iCR(\mu-i\varepsilon)C^*$ converges as $\varepsilon \downarrow 0$, for all $\mu \in K_\lambda$. Again, for $\varepsilon>0$ small enough, $I -iCR_0(\mu -i\varepsilon)C^*$ is invertible in
$\mathcal{L}(\mathcal{H})$. Using \eqref{eq:resolvent2}, we then obtain that
\begin{equation}
iCR(\mu-i\varepsilon)C^* = (I-iCR_0(\mu-i\varepsilon)C^*)^{-1} -I \rightarrow A(\mu)^{-1} -I,
\end{equation}
as $\varepsilon \downarrow 0$. Hence $\mu$ is a regular spectral point of $H$. The fact that the limits
\begin{equation*}
C R ( \mu - i 0 ) C^* = \lim_{ \varepsilon \downarrow 0 } C R ( \mu - i \varepsilon ) C^* = ( I-iCR_0(\mu-i0)C^* )^{-1} - I
\end{equation*}
is uniform in $\mu \in K_\lambda$ follows from the fact that the map $z \mapsto ( I-iCR_0(z)C^* )^{-1}$ is H{\"o}lder continuous on $\{ z \in \mathbb{C} , \, \mathrm{Re}(z) \in K_\lambda , \, \mathrm{Im}(z) \le 0 \}$, by Proposition \ref{prop:standard_csq_2} (i).
\end{proof}

Now we are ready to prove Theorem \ref{thm:invert}.
\begin{proof}[Proof of Theorem \ref{thm:invert}]
Suppose that $V=0$. Recall that the scattering matrices for the pair $(H_0-iC^*C,H_0)$ are denoted by $\tilde{S}(\lambda)$.

First, assume that $\lambda \in \mathring{\Lambda}$ is a regular spectral point of $H$. We must show that $\tilde{S}(\lambda)$ is invertible. By Lemma \ref{lem:criterion}, we know that $I-iCR_0(\lambda-i0)C^*$ is invertible. Using Theorem \ref{thm:RF}, we then compute
\begin{align*}
&\big ( I + 2\pi Z_0(\lambda;C) (I-iCR_0(\lambda-i0)C^*)^{-1}  Z_0(\lambda;C)^* \big )\tilde{S}(\lambda) \\
&= I + 2\pi Z_0(\lambda;C) \big \{ (I-iCR_0(\lambda-i0)C^*)^{-1} - (I-iCR_0(\lambda+i0)C^*)^{-1} \big \} Z_0(\lambda;C)^* \\
& - 4\pi^2 Z_0(\lambda;C) (I-iCR_0(\lambda-i0)C^*)^{-1}  Z_0(\lambda;C)^* Z_0(\lambda;C) (I-iCR_0(\lambda+i0)C^*)^{-1}  Z_0(\lambda;C)^* .
\end{align*}
Recalling that
\begin{align}
Z_0(\lambda;C)^* Z_0(\lambda;C) &= \frac{1}{2i\pi} C[ R_0(\lambda+i0) - R_0(\lambda-i0)]C^* \label{eq:Stone} \\
&= \frac{1}{2\pi}\big \{ ( I -i C R_0(\lambda+i0) C^* ) - ( I - i C R_0(\lambda-i0) C^* ) \big \} , \label{eq:Stone2}
\end{align}
we obtain that $( I + 2\pi Z_0(\lambda;C) (I-iCR_0(\lambda-i0)C^*)^{-1}  Z_0(\lambda;C)^* )\tilde{S}(\lambda) = I$. The same computation shows that $\tilde{S}(\lambda) ( I + 2\pi Z_0(\lambda;C) (I-iCR_0(\lambda-i0)C^*)^{-1}  Z_0(\lambda;C)^* \big ) = I$, and therefore $\tilde{S}(\lambda)$ is indeed invertible (and its inverse is given by \eqref{eq:S^-1}--\eqref{eq:S^-1bis}).

Suppose now that $\lambda \in \mathring{\Lambda}$ is a spectral singularity of $H$. We must show that $\tilde{S}(\lambda)$ is not invertible. We seek a non-vanishing vector $u \in \mathrm{Ker} \, \tilde{S}(\lambda) $ of the form $u = Z_0(\lambda;C) w$, for a suitable $w \in \mathcal{H}$. Using Theorem \ref{thm:RF}, a straightforward calculation shows that the relation $\tilde{S}(\lambda)u=0$ is equivalent to
\begin{equation}
Z_0(\lambda;C) (I-iCR_0(\lambda+i0)C^*)^{-1} A ( \lambda ) w =0 ,
\end{equation}
where, as above, $A ( \lambda )= I-iCR_0(\lambda-i0)C^*$. Since $\lambda$ is a spectral singularity of $H$, $A ( \lambda )$ is not invertible by Lemma \ref{lem:criterion}. Equivalently, since $CR_0(\lambda-i0)C^*$ is compact by Proposition \ref{prop:standard_csq_2} (i), $A ( \lambda )$ is not injective. Let $w \in {\rm{Ker}}\, A ( \lambda )$, $w \not=0$. We claim that $u = Z_0(\lambda;C) w \not =0$. Indeed, assuming by contradiction that $u=0$, we obtain from \eqref{eq:Stone} that
\begin{align*}
(I -iCR_0(\lambda +i0) C^* ) w &= (I-iC R_0(\lambda-i0)C^* ) w - iC (R_0(\lambda +i0)-R_0(\lambda -i0))C^*  w \\
                               &= A ( \lambda )w + 2\pi \Gamma_0(\lambda)^* Z_0(\lambda;C) w \\
                               &= A ( \lambda )w + 2\pi \Gamma_0(\lambda)^* u =0.
\end{align*}
Proposition \ref{prop:standard_csq_2} (ii) shows that $I-iCR_0(\lambda+i0)C^*$ is invertible. Hence the previous equalities give $w=0$, which is a contradiction.

This concludes the proof of Theorem \ref{thm:invert} in the case where $V=0$. As mentioned before, the proof of Theorem \ref{thm:invert} is analogous, using unitarity of the scattering matrices $S_V(\lambda)$ and replacing $\Gamma_0(\lambda)$ by $\Gamma_{\pm}(\lambda)$, and $R_0(\lambda \pm i0)$ by $R_V(\lambda \pm i0)$.
\end{proof}
We conclude this section with a consequence of Lemma \ref{lem:criterion}. It shows, as announced in Remark \ref{rk:defspec-sing} (ii), that the set of spectral singularities of $H$ in $\mathring{\Lambda}$ is closed and has Lebesgue measure $0$.
\begin{proposition}\label{prop:spec-sing}
Suppose that Hypotheses \ref{V-1}--\ref{V-5} hold. Then the set of spectral singularities of $H$ in $\mathring{\Lambda}$ is a closed set whose Lebesgue measure vanishes.
\end{proposition}
\begin{proof}
Let
\begin{align*}
E := \{ \lambda \in \mathring{\Lambda} , \, \lambda \text{ is a spectral singularity of } H \}.
\end{align*}
The fact that $E$ is closed is a direct consequence of Lemma \ref{lem:criterion}. Indeed, it follows from the equivalence (i)$\,\Leftrightarrow\,$(iii) in Lemma \ref{lem:criterion} that the set of regular spectral points of $H$ in $\mathring{\Lambda}$ is open.

To prove that the Lebesgue measure of $E$ is $0$, it suffices to apply Theorem 1.8.3 in \cite{Ya92_01}.
\end{proof}

\subsection{Proof of Theorem \ref{thm:invert_S}}\label{subsec:thminvertS}

In this section, we prove Theorem \ref{thm:invert_S}, using in particular Theorem \ref{thm:invert}.

\begin{proof}[Proof of Theorem \ref{thm:invert_S}]
Suppose that $H$ does not have spectral singularities. It then follows from Theorem \ref{thm:invert} that, for all $\lambda \in \mathring{\Lambda}$, $S(\lambda)^{-1}$ exists and is given by
\begin{align*}
S(\lambda)^{-1} = S_V(\lambda)^{-1} \Big( I + 2\pi Z_V^+(\lambda;C) (I+iCR(\lambda-i0) C^*) Z_V^+(\lambda;C)^* \Big) .
\end{align*}
Using that
\begin{equation*}
Z_V^+(\lambda;C)^* Z_V^+(\lambda;C) = \frac{1}{2i\pi} C \big ( R_V( \lambda + i 0 ) - R_V( \lambda - i 0 ) \big ) C^* ,
\end{equation*}
together with Hypothesis \ref{V-5}, we see that $\mathring{\Lambda} \ni \lambda \mapsto Z_V^+(\lambda;C)$ is bounded. Moreover, since $\Lambda \setminus \mathring{\Lambda}$ is finite and since all $\lambda \in \Lambda \setminus \mathring{\Lambda}$ are regular in the sense of Definition \ref{def:spec-sing}, we deduce that $\mathring{\Lambda} \ni \lambda \mapsto CR(\lambda-i0) C^*$ is also bounded. By Proposition \ref{prop:standard_csq} (viii), $S_V(\lambda)$ is unitary. Therefore $\mathring{\Lambda} \ni \lambda \to S(\lambda)^{-1}$ is bounded. This implies that $S(H,H_0)$ is invertible with inverse given by
\begin{equation*}
S(H,H_0)^{-1} = \mathcal{F}_0^* \bigg( \int_\Lambda^\oplus S(\lambda)^{-1} d\lambda \bigg) \mathcal{F}_0.
\end{equation*}

Suppose now that $H$ has a spectral singularity $\lambda_0 \in \mathring{\Lambda}$. By Theorem \ref{thm:invert}, we know that $S(\lambda_0)$ is not invertible in $\mathcal{M}$. Since $S(\lambda_0)-I$ is compact by Theorem \ref{thm:RF}, this means that there exists $u_{\lambda_0} \in \mathcal{M}$, $\| u_{\lambda_0} \|_{\mathcal{M}}Ê= 1$, such that $S(\lambda_0) u_{\lambda_0} = 0$. Let $\varepsilon > 0$. Since $\lambda \mapsto S(\lambda)$ is continuous by Theorem \ref{thm:RF}, there exists $\delta>0$ such that, for all $\lambda \in [ \lambda_0 - \delta , \lambda_0 + \delta ]$, $\| S( \lambda ) - S( \lambda_0 ) \| \le \varepsilon$. Let $h \in \mathrm{C}_0^\infty( [ - \delta , \delta ] ; \mathbb{C} )$ be a function such that $\| h \|_2 = 1$ and let $u_\varepsilon \in L^2( \Lambda ; \mathcal{M} )$ be defined by
\begin{equation*}
u_\varepsilon(\lambda) = h( \lambda - \lambda_0 ) u_{\lambda_0}.
\end{equation*}
Note that $\| u_\varepsilon \|_{L^2(\Lambda ; \mathcal{M} ) }  = 1$. We compute
\begin{align*}
\big \|ÊS(H,H_0) \mathcal{F}_0^* u_\varepsilon \big \|^2_{\mathcal{H}}  &= \int_\Lambda \big \| h( \lambda - \lambda_0 ) S(\lambda) u_{\lambda_0} \big \|_{\mathcal{M}}^2 d\lambda \\
& \le \int_{ [ \lambda_0 - \delta , \lambda_0 + \delta ] } | h ( \lambda - \lambda_0 ) |^2 \big \| ( S( \lambda ) - S( \lambda_0 ) ) u_{\lambda_0} \big \|_{\mathcal{M}}^2 d \lambda ,
\end{align*}
where we used that $S(\lambda_0)u_{\lambda_0}=0$ and that $h$ is supported in $[-\delta , \delta]$. Since $\| u_{\lambda_0} \|_{\mathcal{M}}=1$ and $\| S(\lambda) - S(\lambda_0) \| \le \varepsilon$ for $\lambda \in [ \lambda_0 - \delta , \lambda_0 + \delta ]$, this implies that
\begin{align*}
\big \|ÊS(H,H_0) \mathcal{F}_0^* u_\varepsilon \big \|^2_{\mathcal{H}}  &\le \varepsilon^2 \int_{ [ \lambda_0 - \delta , \lambda_0 + \delta ] } | h ( \lambda - \lambda_0 ) |^2 d \lambda = \varepsilon^2 ,
\end{align*}
since $\| h \|_2 = 1$. Hence we have proven that for all $\varepsilon>0$, there exists a normalized vector $v_\varepsilon = \mathcal{F}_0^* u_\varepsilon$ such that $\|ÊS( H , H_0 ) v_\varepsilon \| \le \varepsilon$. This shows that $S(H,H_0)$ is not injective with closed range, hence not invertible.

The fact that
\begin{align*}
Ê\mathrm{Ran}( W_-( H , H_0 ) ) = \big ( \mathcal{H}_{ \mathrm{b} }( H ) \oplus \mathcal{H}_{ \mathrm{d} }( H^* ) \big )^\perp,
\end{align*}
if $S(H,H_0)$ is invertible, is proven in \cite[Proposition 3.8]{FaFr18_01}.
\end{proof}

\vspace{0.3cm}\noindent
{\it{Acknowledgements:}}  We would like to warmly thank Xue Ping Wang for his kind encouragement and useful discussions. We also thank Serge Richard for useful comments. J.F. is grateful to J{\"u}rg Fr{\"o}hlich for many discussions and collaborations.

\end{document}